\documentclass[11pt,a4paper]{article}
\pdfoutput=1

\usepackage[latin1]{inputenc}
\usepackage[english]{babel}
\usepackage[left=2.7cm, right=2.7cm, top=2.7cm, bottom=2.7cm]{geometry}

\usepackage{authblk}

\usepackage[titletoc,title]{appendix}

\usepackage{amsmath}
\usepackage{amsfonts}
\usepackage{amssymb}
\usepackage{amsthm}
\usepackage{bbm}		
\usepackage{mathtools}	
\usepackage{xcolor}
\usepackage{tikz}

\usepackage{graphicx}
\usepackage[suffix=]{epstopdf}
\usepackage{booktabs}
\usepackage[font=small,labelfont=bf]{caption} 

\definecolor{thelinkcolor}{RGB}{0,0,150}
\usepackage{hyperref}
\hypersetup{colorlinks=true,
			urlcolor=blue,
			linkcolor=blue,
			citecolor=blue,
			bookmarksdepth=paragraph,
			pdftitle={Bounding extreme events in nonlinear dynamics (arXiv version)},
			pdfauthor={G. Fantuzzi and D. Goluskin},}
\usepackage[nameinlink]{cleveref}
\crefname{equation}{}{Equation}
\crefname{theorem}{Theorem}{Theorem}
\crefname{example}{example}{Example}
\crefname{lemma}{Lemma}{Lemma}
\crefname{proposition}{Proposition}{Proposition}
\crefname{figure}{figure}{Figure}
\crefname{table}{table}{Table}
\usepackage{lipsum}
\usepackage{amsfonts}
\usepackage{amssymb}
\usepackage{graphicx}
\usepackage{epstopdf}
\usepackage{algorithmic}
\usepackage{tikz}
\usepackage{bm}
\usepackage{mathtools}	
\usepackage{hyperref}
\usepackage{xcolor}

\newcommand{\R}{\mathbb{R}}				
\renewcommand{\vec}[1]{\mathbf{#1}} 	
\newcommand{\Rey}{\mbox{\it Re}} 		
\newcommand{\dt}{\,{\rm d}t} 			
\newcommand{\ssep}{\,;}
\newcommand{\abs}[1]{\left\vert #1 \right\vert}

\definecolor{matlabblue}{rgb}{0,0,1}
\definecolor{matlabred}{rgb}{1,0,0}
\definecolor{matlabgreen}{rgb}{0,0.5,0}
\definecolor{matlaborange}{rgb}{0.8706,0.4902,0}
\definecolor{matlabmagenta}{rgb}{0.7490,0,0.7490}
\definecolor{matlabgrey}{rgb}{0.65,0.65,0.65}

\newcommand\solidrule[1][10pt]{\rule[0.5ex]{#1}{1.5pt}}

\newcommand{\mysquare}[1]{%
	\protect\begin{tikzpicture}%
	\protect\draw[thick,color=#1] (0,0) -- (0.75ex,0) -- (0.75ex,0.75ex) -- (0,0.75ex) -- (0,0);
	\protect\end{tikzpicture}%
}
\newcommand{\mytriangle}[1]{%
	\protect\begin{tikzpicture}%
	\protect\draw[thick,color=#1] (0,0) -- (1ex,0) -- (0.5ex,0.866ex) -- (0,0);
	\protect\end{tikzpicture}%
}
\newcommand{\myudtriangle}[1]{%
	\protect\begin{tikzpicture}%
	\protect\draw[thick,color=#1] (0,0.866ex) -- (1ex,0.866ex) -- (0.5ex,0ex) -- (0,0.866ex);
	\protect\end{tikzpicture}%
}
\newcommand{\mycross}[1]{%
	\protect\begin{tikzpicture}%
	\protect\draw[thick,color=#1] (0,0) -- (1ex,1ex);
	\protect\draw[thick,color=#1] (0,1ex) -- (1ex,0);
	\protect\end{tikzpicture}%
}
\newcommand{\mycircle}[1]{%
	\protect\begin{tikzpicture}%
	\protect\draw[thick,color=#1] (0.5ex,0.5ex) circle (0.5ex);
	\protect\end{tikzpicture}%
}

\makeatletter
\newcommand{\subalign}[1]{%
	\vcenter{%
		\Let@ \restore@math@cr \default@tag
		\baselineskip\fontdimen10 \scriptfont\tw@
		\advance\baselineskip\fontdimen12 \scriptfont\tw@
		\lineskip\thr@@\fontdimen8 \scriptfont\thr@@
		\lineskiplimit\lineskip
		\ialign{\hfil$\m@th\scriptstyle##$&$\m@th\scriptstyle{}##$\crcr
			#1\crcr
		}%
	}
}
\makeatother

\ifpdf
\DeclareGraphicsExtensions{.eps,.pdf,.png,.jpg}
\else
\DeclareGraphicsExtensions{.eps}
\fi

\usepackage{enumitem}
\setlist[enumerate]{leftmargin=.5in}
\setlist[itemize]{leftmargin=.5in}

\newtheorem{theorem}{Theorem}

\theoremstyle{definition}

\usepackage{amsopn}

\usepackage[sort&compress,numbers]{./natbib}
\renewcommand{\bibfont}{\small}
\author[1]{Mayur V. Lakshmi\thanks{Email address for correspondence: \href{mailto:mayur.lakshmi11@imperial.ac.uk}{mayur.lakshmi11@imperial.ac.uk}}}
\author[1]{Giovanni Fantuzzi}
\author[2]{Jes\'{u}s D. Fern\'{a}ndez-Caballero}
\author[1]{Yongyun Hwang}
\author[1]{Sergei I. Chernyshenko}
\affil[1]{Department of Aeronautics, Imperial College London, London, SW7 2AZ, United Kingdom.}
\affil[2]{Mathematics Institute, Zeeman Building, University of Warwick, Coventry, CV4 7AL, UK}
\title{Finding extremal periodic orbits with polynomial optimization, with application to a nine-mode model of shear flow}

\begin{document}
\maketitle

\vspace*{-5ex}
\begin{abstract}\noindent
	Tobasco \textit{et al.} [\href{https://doi.org/10.1016/j.physleta.2017.12.023}{\textit{Phys. Lett. A}, 382:382--386, 2018}] recently suggested that trajectories of ODE systems \textcolor{black}{that} optimize the infinite-ti5e average of a certain observable can be localized using sublevel sets of a function that arise when bounding such averages using so-called auxiliary functions. In this paper we demonstrate that this idea is viable and allows for the computation of extremal unstable periodic orbits (UPOs) for polynomial ODE systems. First, we prove that polynomial optimization is guaranteed to produce auxiliary functions that yield near-sharp bounds on time averages, which is required in order to localize the extremal orbit accurately. Second, we show that points inside the relevant sublevel sets can be computed efficiently through direct nonlinear optimization. Such points provide good initial conditions for UPO computations. We then combine these methods with a single-shooting Newton--Raphson algorithm to study extremal UPOs for a nine-dimensional model of sinusoidally forced shear flow. We discover three previously unknown families of UPOs, one of which simultaneously minimizes the mean energy dissipation rate and maximizes the mean perturbation energy relative to the laminar state for Reynolds numbers approximately between 81.24 and 125.
\end{abstract}



\section{Introduction}
\label{s:intro}
In the study of dynamical systems governed by nonlinear differential equations, the problem of finding unstable periodic orbits (UPOs) is of significant interest. The set of UPOs forms a skeleton for the chaotic dynamics of many nonlinear dissipative systems~\cite{Cvitanovic1991a} and is sometimes dense in the chaotic attractor~\cite{Eckmann1985}. Therefore, the identification of UPOs is useful to understand the topology of the attractor. In addition, knowledge of UPOs and their respective stability eigenvalues can be leveraged to calculate the infinite-time average of any observable of interest over chaotic trajectories through a cycle expansion of the corresponding dynamical zeta function~\cite{Artuso1990,Cvitanovic1995}. A detailed treatment of cycle expansions can be found in~\cite{Cvitanovic2017}. Finally, UPOs also play a key role in the control of chaos~\cite{Ott1990}.

One of the main challenges in the computation of UPOs is that the convergence of the Newton--Raphson algorithm applied to the flow map~\cite[Ch. 7]{Cvitanovic2017},~\cite{Kelley2003} relies on the availability of good initial guesses for both the UPO (or at least one point on it) and its period. A more sophisticated algorithm for computing UPOs was derived by Viswanath \cite{Viswanath2001}, and is based upon the Lindstedt--Poincar\'{e} technique. However, sufficiently good initial guesses for the UPO and its period are also required for this algorithm to ensure convergence.

One popular method to search for such initial guesses was introduced by Auerbach \textit{et al.}~\cite{Auerbach1987}, who suggested to look for near periodic sections of simulated chaotic trajectories, which signal shadowing of a UPO. Another approach exists for systems with at least two local attractors; see, for instance,~\cite{Skufca2006,Kim2008,Joglekar2015}. The idea behind this approach is that initial conditions on the boundary of the local attraction basins converge to neither attractor, but rather to one of possibly many unstable ``edge states", whose stable manifolds are embedded in the basin boundary. Using a simple bisection algorithm, one can obtain a point lying very close to the basin boundary, such that its forward trajectory approaches very closely an edge state---very often, a UPO. The aforementioned methods have proved effective in practice and are applicable to very-high-dimensional systems. However, there is little control on which UPOs can be identified.

In this work, we consider the problem of computing UPOs that are \textit{extremal}, in the sense that they maximize or minimize the infinite-time average of a certain observable. Such orbits are especially interesting for control purposes, since their knowledge allows one to design control strategies to stabilize desired dynamics or suppress undesired ones. Observe also that every periodic orbit is extremal for at least one observable, such as its own indicator function (equal to one on the periodic orbit and zero elsewhere). Thus, studying extremal UPOs is not restrictive. Of course, indicator functions can not be used in this context because their calculation (or approximation) requires a priori knowledge of the corresponding UPO, which is precisely what one wishes to compute. Nevertheless, varying the observable whose time average is to be optimized still allows one to systematically probe the state space and potentially identify a large number of periodic orbits.

The identification of extremal UPOs is inherently linked to the study of infinite-time averages of scalar-valued observables, as these are thought to be typically maximized (or minimized) on periodic orbits~\cite{Yang2000}. For systems governed by ordinary differential equations (ODEs), extremal time averages can be bounded rigorously to arbitrary accuracy using so-called auxiliary functions of the system's state, which are similar to but distinct from the Lyapunov functions used in nonlinear stability analysis. This strategy, proposed by Chernyshenko \textit{et al.}~\cite{Chernyshenko2014} and further considered in~\cite{Fantuzzi2016,Goluskin2018a,Goluskin2019}, generalizes the ``background method'' developed by Doering \& Constantin in fluid mechanics~\cite{Doering1992,Doering1994,Constantin1995,Doering1996}, which amounts to using a quadratic auxiliary function~\cite{Chernyshenko2017}. A dual formulation of the same approach, in the sense of convex duality, also exists (see, e.g.~\cite[Section 6]{Korda2018a} and~\cite{Tobasco2018}) and is based on the well-known connection between infinite-time averages and invariant measures~\cite[Lectures 1 and 6]{Sinai1977}.

Auxiliary functions are attractive for two reasons. First, the optimization of auxiliary functions is a convex problem, while directly optimizing time averages over trajectories is a nonconvex and nonlinear problem. Second, for ODE systems with polynomial dynamics, polynomial auxiliary functions can be constructed and optimized numerically using methods for polynomial optimization. Indeed, the constraint to be satisfied by any candidate auxiliary function becomes a non-negativity condition on a polynomial. This is an NP-hard constraint in general~\cite{Murty1987} but can be strengthened into the condition that the polynomial admits a sum-of-squares (SOS) decomposition. Problems with SOS constraints can be posed as semidefinite programs (SDPs) and can be solved computationally using algorithms with polynomial-time complexity~\cite{Vandenberghe1996}. While SOS constraints are only sufficient and not necessary for non-negativity (except for certain particular classes of polynomials; see, e.g.~\cite{Reznick2000}), experience has shown that they work extremely well in practice. Further, in certain cases powerful results on the existence of SOS representations for polynomials (e.g.~\cite{Putinar1993}) guarantee that replacing non-negativity with SOS constraints is not restrictive. For these reasons, many authors have used auxiliary functions, generalizations thereof, and their measure-theoretic counterparts (e.g. the aforementioned invariant measures) to study dynamical systems. Applications beyond time averages include nonlinear stability analysis~\cite{Papachristodoulou2005,Goulart2012,Chernyshenko2014,Huang2015}, estimating regions of attraction~\cite{Henrion2014,Tan2006a}, bounding extreme values on attractors \cite{Goluskin2018}, nonlinear control~\cite{Lasagna2016,Lasserre2008,Majumdar2014}, ergodic optimization of discrete-time dynamical systems and ODEs~\cite{Henrion2012,Korda2018a}, and bounding stationary expectations for discrete-time and continuous-time stochastic processes~\cite{Henrion2012,Korda2018a,Kuntz2016,Fantuzzi2016}.

The connection between auxiliary functions proving bounds on infinite-time averages and the corresponding extremal trajectories, which are often UPOs, has been recently established by Tobasco \textit{et al.}~\cite{Tobasco2018}. Precisely, if an auxiliary function produces nearly sharp bounds, then it can be used to construct a function whose sublevel sets localize extremal and near-extremal trajectories in state space. Here we leverage this observation and propose a new approach to compute extremal UPOs for ODE systems with polynomial dynamics. The first step in this approach is to construct near-optimal auxiliary functions using SOS optimization. We prove that this is always possible with a small modification to the computational methods described in the previous works~\cite{Chernyshenko2014,Fantuzzi2016,Goluskin2018a}. This follows from the argument in~\cite{Tobasco2018} by combining polynomial approximations of continuously differentiable functions on compact sets with a standard argument in SOS optimization. Our result is the dual version of Theorem 2 in~\cite{Korda2018a} in the context of bounding infinite-time averages, although the latter result applies to a broader class of optimization problems over invariant measures. The second step is to compute points in the $\varepsilon$-sublevel set of the function constructed using the near-optimal auxiliary function for small $\varepsilon$, and use these as initial conditions for converging to UPOs with existing algorithms. We demonstrate that while full sublevel sets of multivariate polynomials in high-dimensional spaces cannot be computed in practice, good initial conditions can be obtained with relatively inexpensive nonlinear optimization methods such as quasi-Newton methods.

As a proof of concept, we apply our new method to a nine-dimensional model of sinusoidally forced shear flow~\cite{Moehlis2004}. We consider this particular system for two reasons. First, it is sufficiently low dimensional for the aforementioned SOS optimization approach to be computationally feasible, but it is also sufficiently high dimensional that approximating extremal UPOs by computing full sublevel sets of polynomials as done in~\cite{Tobasco2018} is not possible. Second, lower bounds on the mean energy dissipation rate in the model were computed by optimizing polynomial auxiliary functions \cite{Chernyshenko2014}. For values of the Reynolds number $\Rey$ larger than 40, these bounds appear to converge to a value strictly smaller than both the dissipation rate of the laminar flow state and the simulated dissipation rate of the turbulent flow state as the degree of the polynomial auxiliary function is raised. It was hypothesized in \cite{Chernyshenko2014} that this behavior is due to the existence of some periodic orbit which saturates the bounds. While a relatively large number of families of UPOs have already been found for this flow model~\cite{Moehlis2005}, there remains the possibility that other UPOs have thus far eluded researchers.

Here we investigate whether optimizing polynomial auxiliary functions of higher degree improves the bounds reported in \cite{Chernyshenko2014}, and whether for $\Rey > 40$ the mean energy dissipation rate is indeed minimized by UPOs. Using the methods described above, we obtain approximations to the UPOs that minimize the infinite-time average of the energy dissipation rate, as well as to those maximizing the infinite-time average of the energy of perturbations from the laminar flow state. We use our approximations to find three new families of UPOs which do not appear to be among those reported in~\cite{Moehlis2005}.

The rest of this work is structured as follows. \Cref{s:optimal-bounds} reviews the auxiliary function method to bound extremal infinite-time averages. In \cref{s:po-approximation} we describe our approach to localizing UPOs using auxiliary functions. \Cref{s:sos-bounds} explains how to construct near-optimal auxiliary functions using SOS optimization and includes theoretical results on the convergence of the numerical bounds to the extremal time average, as well as on the exploitation of symmetries in computations. Readers interested in computing UPOs may skip this section at first reading. \Cref{s:shear-flow-results} reports the results obtained when applying our methods to the nine-mode shear flow model from~\cite{Moehlis2004}. Finally, \cref{s:conclusion} suggests possible future research directions and offers concluding remarks.

\section{Infinite-time averages and auxiliary functions}
\label{s:optimal-bounds}
Bounding infinite-time averages using auxiliary functions is the first step in our approach to computing extremal UPOs.
Consider an autonomous dynamical system governed by the ODE
\begin{equation}
\label{e:ode}
\frac{\mathrm{d}\vec{a}}{\dt}  = \vec{f}(\vec{a}), \quad \vec{a}(0) =\vec{a}_0,
\end{equation}
where $\vec{a} \in \mathbb{R}^n$. Following~\cite{Tobasco2018}, we assume that trajectories exist at all times and are continuously differentiable with respect to their initial conditions. This is true, for example, when $\vec{f} : \mathbb{R}^n \mapsto \mathbb{R}^n$ is continuously differentiable.  We also assume that there exists a compact set $\Omega$ in which all solutions $\vec{a} (t \ssep \vec{a}_0 )$ of~\cref{e:ode} eventually remain, irrespective of the initial conditions $\vec{a}_0$. Such a set may be found in a variety of ways, but one very common approach is to let $\Omega = \{\vec{a} \in \R^n \mid\, W(\vec{a})\leq C\}$, where $W:\R^n \to \R$ is a radially unbounded and continuously differentiable function such that
\begin{equation}
\label{e:absorbing-set-conditions}
\lambda\,  \vec{f}(\vec{a}) \cdot \nabla W(\vec{a}) \leq C - W(\vec{a}) \qquad \forall \vec{a}\in\R^n,
\end{equation}
for some scalars $C$ and $\lambda>0$. While these conditions produce a globally absorbing set (this follows from Gronwall's inequality; see, e.g.~\cite{Goluskin2018,Goluskin2019}), we need not assume that $\Omega$ is absorbing; trajectories may transiently leave $\Omega$, provided that they eventually re-enter and remain in it. In particular, any closed ball containing a globally absorbing set is a suitable choice for $\Omega$.

Given a continuous function $\Phi : \mathbb{R}^n \to \mathbb{R}$, which represents an observable of interest, we define its infinite-time average along the trajectory starting from $\vec{a}_0$ as
\begin{equation}
\label{e:phi-time-average-definition}
\overline{\Phi} \left(\vec{a}_0\right) := \limsup_{T \to \infty} \frac{1}{T} \int_0^T \Phi \left[ \vec{a}\left(t \ssep \vec{a}_0\right) \right] dt.
\end{equation}
We use the lim sup in this definition because time averages need not converge. As already noted in~\cite{Tobasco2018}, $\overline{\Phi}(\vec{a}_0)$ could alternatively be defined using the lim inf, with no effect on the results presented in this work. Conditions under which time averages do converge are discussed in~\cite{Karabacak2011}.

We are interested in the maximal value of $\overline{\Phi}$ over all trajectories,
\begin{equation}
\label{e:max-phi-definition}
\overline{\Phi}^* := \max _{\vec{a}_0 \in \Omega} \overline{\Phi}(\vec{a}_0),
\end{equation}
as well as the initial conditions and corresponding trajectories which achieve it. The optimization problem on the right-hand side of \cref{e:max-phi-definition} is well posed and there exists an optimal initial condition $\vec{a}_0^*$ such that $\overline{\Phi}(\vec{a}_0^*) = \overline{\Phi}^*$ (see, e.g.~\cite{Tobasco2018}); in fact, there are clearly infinitely many such optimal initial conditions because $\overline{\Phi}[\vec{a}(t\ssep \vec{a}_0^*)] = \overline{\Phi}(\vec{a}_0^*)$ for any fixed time $t$. Observe that considering maximal time averages only is not restrictive because the minimal value of $\overline{\Phi}$ coincides with the maximum of $\overline{-\Phi}$. Additionally, there is no loss of generality in assuming that the initial condition belongs to $\Omega$ because every trajectory enters and remains in it after some finite time, and the part of the trajectory before this time gives no contribution to $\overline{\Phi}(\vec{a}_0)$.

Upper bounds on $\overline{\Phi}^*$ can be proven using an approach originally proposed in~\cite{Chernyshenko2014} and further studied in~\cite{Fantuzzi2016,Goluskin2018a,Goluskin2019,Lasagna2016}. The method relies on the simple observation that, for all initial conditions $\vec{a}_0$ in $\Omega$ and any function $V: \Omega \to \mathbb{R}$ in the class $C^1(\Omega)$ of continuously differentiable functions on $\Omega$ (hereafter called \textit{auxiliary function}),
\begin{equation}
\overline{\vec{f}[\vec{a}(t)]\cdot \nabla V[\vec{a}(t)]}
= \overline{\frac{\mathrm{d}}{\dt} V[\vec{a}(t)]}
= \lim_{T \to \infty} \frac{ V[\vec{a}(T)] - V(\vec{a}_0)}{T}
= 0.
\end{equation}
This immediately implies $\overline{\Phi}(\vec{a}_0) = \overline{\Phi + \vec{f} \cdot \nabla V}(\vec{a}_0)$, i.e., the functions $\Phi$ and $\Phi + \vec{f}\cdot \nabla V$ have the same infinite-time average. Moreover, since trajectories eventually remain in the compact set $\Omega$ we have the pointwise bound
\begin{equation}
\overline{\Phi + \vec{f}\cdot \nabla V}(\vec{a}_0) \leq  
\max_{\vec{a} \in \Omega} \left\{ \Phi(\vec{a}) + \vec{f}(\vec{a})\cdot \nabla V(\vec{a}) \right\}
\end{equation} 
irrespective of the initial condition $\vec{a}_0$. Thus,
\begin{equation}
\overline{\Phi}(\vec{a}_0) \leq \max_{\vec{a} \in \Omega} \left\{ \Phi(\vec{a}) + \vec{f}(\vec{a})\cdot \nabla V(\vec{a}) \right\}.
\end{equation}
Upon minimizing the right-hand side of this inequality over all auxiliary functions $V$ in $C^1(\Omega)$, and subsequently maximizing the left-hand side over $\vec{a}_0$, we obtain an upper bound on $\overline{\Phi}^*$,
\begin{equation}
\overline{\Phi}^* \leq \adjustlimits \inf_{V \in C^1(\Omega)} \max_{\vec{a} \in \Omega} 
\left\{ \Phi(\vec{a}) + \vec{f}(\vec{a})\cdot \nabla V(\vec{a}) \right\}.
\end{equation}
In fact, Tobasco \textit{et al.}~\cite{Tobasco2018} recently proved that, under all assumptions on the trajectories of \cref{e:ode} and on $\Omega$ outlined above, this inequality is actually an \textit{equality}. In other words, auxiliary functions characterize extremal infinite-time averages exactly:
\begin{equation}
\label{e:inf-sup-V}
\overline{\Phi}^* = \adjustlimits \inf_{V \in C^1(\Omega)} \max_{\vec{a} \in \Omega} 
\left\{ \Phi(\vec{a}) + \vec{f}(\vec{a})\cdot \nabla V(\vec{a}) \right\}.
\end{equation}


The key feature of the minimax problems in~\cref{e:inf-sup-V} is that optimizing auxiliary functions to obtain good upper bounds on $\overline{\Phi}^*$ does \textit{not} require solving the ODE~\cref{e:ode}, but rather estimating the global maximum of $\Phi + \vec{f}\cdot \nabla V$ over $\Omega$. For highly nonlinear or chaotic systems, therefore, finding nearly sharp bounds may be easier than a direct calculation of $\overline{\Phi}^*$ via extensive numerical simulations. In addition, while optimizing fully general $V$ is not easy, very good bounds can be obtained in practice by considering subsets of auxiliary functions that can be optimized numerically. This has already been clearly demonstrated by the results in~\cite{Goluskin2018a,Goluskin2019}, where nearly optimal polynomial auxiliary functions for the Lorenz system and the Kuramoto--Sivashinsky equation were constructed numerically using SOS optimization. In \cref{s:sos-bounds} we show that if $\vec{f}$ and $\Phi$ are polynomials and $\Omega$ is a compact semialgebraic set subject to a mild technical condition, then arbitrarily sharp bounds on $\overline{\Phi}^*$ and the corresponding near-optimal auxiliary functions can, at least in principle, always be constructed numerically using a variation of the computational approach utilized in~\cite{Chernyshenko2014,Fantuzzi2016,Goluskin2018a,Goluskin2019}.

\section{Approximation of extremal trajectories}
\label{s:po-approximation}
In addition to yielding arbitrarily sharp bounds on $\overline{\Phi}^*$, auxiliary functions can be used to localize the associated extremal trajectories in state space. To see this, suppose that a fully optimal auxiliary function $V^*$, which yields a bound $\lambda$ exactly equal to $\overline{\Phi}^*$, exists and is available. Then, the extremal trajectory $\vec{a}(t)$ must satisfy~\cite{Fantuzzi2016,Tobasco2018}
\begin{equation}
\overline{\lambda - \left( \vec{f}[\vec{a}(t)] \cdot \nabla V^*[\vec{a}(t)] + \Phi[\vec{a}(t)] \right)}= 0
\end{equation}
with $\lambda = \overline{\Phi}^*$. Since the quantity being averaged is nonnegative, if the extremal trajectory is periodic it must lie inside the set
\begin{equation}
\mathcal{S}_0 := \{ \vec{a} \in \Omega : \lambda - \left[ \vec{f}(\vec{a}) \cdot \nabla V^*(\vec{a}) + \Phi(\vec{a}) \right] = 0 \}.
\end{equation}
While not all points in $\mathcal{S}_0$ necessarily belong to the extremal trajectory, they provide guidance to locate it.

Unfortunately, an optimal auxiliary function may not exist because the infimum in~\cref{e:inf-sup-V} need not be attained. When it does exist, moreover, it may not be computable. Nevertheless, equation~\cref{e:inf-sup-V} implies that for any $\delta > 0$ there exists a $\delta$-suboptimal auxiliary function. More precisely, for all $\delta>0$ there exists a $V \in C^1(\Omega)$ which provides a bound $\lambda$ with $\overline{\Phi}^* \leq \lambda \leq \overline{\Phi}^* + \ \delta$. Furthermore, we show in the next section that the equality \cref{e:inf-sup-V} still holds if one optimizes instead over the space $\Pi_n$ of $n$-variate polynomials (see \cref{e:inf-sup-V-poly}). The implication is that for all $\delta>0$, there exists a $\delta$-suboptimal polynomial auxiliary function. We also show in the next section that such $V$ can always be computed numerically, although doing so for high-dimensional systems might require prohibitively large computational resources. As described in~\cite{Tobasco2018}, any suboptimal $V$ constructed numerically can also be used to approximately localize the extremal trajectory, although the results are necessarily weaker compared to the case in which $V$ is optimal. Specifically, for any pair $(\lambda,V)$, where $\lambda$ is a $\delta$-suboptimal bound on $\overline{\Phi}^*$ and $V$ is the corresponding $\delta$-suboptimal polynomial auxiliary function, let $\mathcal{P}_{\lambda,V}$ denote the nonnegative polynomial
\begin{equation}
\label{e:polynomial}
\mathcal{P}_{\lambda,V}(\vec{a}) := \lambda - \left[ \vec{f}(\vec{a}) \cdot \nabla V(\vec{a}) + \Phi(\vec{a}) \right]
\end{equation}
and consider the set of all points where $\mathcal{P}_{\lambda,V}$ is no greater than some arbitrary $\varepsilon>0$:
\begin{equation}
\label{e:tobasco-set}
\mathcal{S}_{\varepsilon} = \{ \vec{a} \in \Omega : \mathcal{P}_{\lambda,V} \leq \varepsilon \}.
\end{equation}
For any $\delta$-suboptimal $V$, the extremal trajectory is guaranteed to lie in $\mathcal{S}_{\varepsilon}$ for a fraction of time determined by the values of $\varepsilon$ and $\delta$. In particular, if the extremal trajectory is a periodic orbit, the fraction of its time period spent inside the set $\mathcal{S}_{\varepsilon}$ is no smaller than $F:=1- \delta/\varepsilon$~\cite[Section 3]{Tobasco2018}. Since $\mathcal{P}_{\lambda,V}(\vec{a}) \geq 0$ on the compact set $\Omega$, the volume of $\mathcal{S}_{\varepsilon}$ is small for sufficiently small $\varepsilon$ if $\mathcal{P}_{\lambda,V}$ is not a constant polynomial.\footnote{Let ${\mathcal{L}^n}(A)$ denote the $n$-dimensional Lebesgue measure (i.e. volume) of a set $A \subset \R^n$ and fix any $\delta > 0$ arbitrarily small. Observe that ${\mathcal{L}^n}(\mathcal{S}_0)=0$ because $\mathcal{S}_0$ is the zero level set of a nonconstant polynomial. Further, ${\mathcal{L}^n}(\mathcal{S}_\varepsilon)$ is finite for all $\varepsilon$ because $\mathcal{S}_{\varepsilon}$ is a closed subset of the compact set $\Omega$, so it is itself compact. For any sequence $\{\varepsilon_k\}_{k \in \mathbb{N}}$ converging to zero we have $\bigcap \mathcal{S}_{\varepsilon_k} = \mathcal{S}_0$. The continuity of the Lebesgue measure gives ${\mathcal{L}^n}(\mathcal{S}_{\varepsilon_k}) \to {\mathcal{L}^n}(\mathcal{S}_0) = 0$, so there exists $\varepsilon_k$ such that ${\mathcal{L}^n}(\mathcal{S}_{\varepsilon_k}) \leq \delta$.} When $\delta \ll 1$, we can have $F$ close to 1 for $\varepsilon$ not too large, so $\mathcal{S}_{\varepsilon}$ can be expected to approximate the location of extremal and near-extremal trajectories. This has already been demonstrated in~\cite{Tobasco2018} for the Lorenz system~\cite{Lorenz1963}.

Even when a polynomial auxiliary function $V$ producing a near-optimal bound $\lambda$ on $\overline{\Phi}^*$ is available, however, approximating extremal trajectories using the full set $\mathcal{S}_\varepsilon$ for a given $\varepsilon$ is computationally intractable except for very-low-dimensional ODE systems. When simply ``gridding'' the state space and evaluating $\mathcal{P}_{\lambda,V}$ at the grid points to approximate $\mathcal{S}_\varepsilon$ is not viable, a simpler strategy to compute points in $\mathcal{S}_\varepsilon$ is to numerically look for points where $\mathcal{P}_{\lambda,V}(\vec{a})\leq \varepsilon$ using any nonlinear minimization algorithm, initialized using random initial conditions $\vec{a}_0 \in \Omega$. {Since $\mathcal{P}_{\lambda,V}$ is non-convex, for each choice of $\vec{a}_0$ this minimization typically returns a local minimizer $\vec{a}^*$. If $\mathcal{P}_{\lambda,V}(\vec{a}^*) \leq \varepsilon$, then $\vec{a}^*$ belongs to $\mathcal{S}_\varepsilon$. Otherwise, the minimization should be repeated starting from a different random initial condition.}

{The initial condition for the nonlinear minimization algorithm can be sampled from any distribution. For simplicity, in this work we use the uniform distribution, but better strategies are possible. One possible computationally efficient approach is to first choose a random $\vec{a}_0 \in \Omega$ from any given probability distribution, and then evaluate $e^{-\beta\mathcal{P}_{\lambda,V}(\vec{a_0})}$ for some parameter $\beta > 0$. This expression takes on a value close to 1 only if $\mathcal{P}_{\lambda,V}\vec{a}_0 \approx 0$, meaning that $\vec{a}_0$ is a near minimizer for $\mathcal{P}_{\lambda,V}$ and, therefore, is likely to be a good initial conditions from which to start the nonlinear minimization algorithm.} 

{Another} interesting observation is that, when $\varepsilon$ is small, all points in $\mathcal{S}_\varepsilon$ along the extremal trajectory are {almost minimizers} for $\mathcal{P}_{\lambda,V}$. Thus, the polynomial $\mathcal{P}_{\lambda,V}$ will be relatively ``flat'' along the part of the extremal trajectory contained in $\mathcal{S}_\varepsilon$, and steeper elsewhere. It is not unreasonable to expect that the minimization routine used would quickly descend to this flat region and then slowly progress towards {a minimizer} for $\mathcal{P}_{\lambda,V}$. In the process, one can obtain a sequence of points for which $\mathcal{P}_{\lambda,V}(\vec{a})\leq \varepsilon$ and that, crucially, ``shadow'' part of the extremal trajectory. {If $V$ is near-optimal, meaning that a large portion of the extremal trajectory lies in the set $\mathcal{S}_\varepsilon$, tracking this sequence of points can produce a richer approximation of the extremal trajectory.} In light of this, one should choose a minimization algorithm which is unlikely to stall, but nevertheless converges slowly to local minima in order to produce many points in $\mathcal{S}_{\varepsilon}$. Typical algorithms which fit these requirements are variants of the quasi-Newton method.

An important complication is that only a finite number of local minima of $\mathcal{P}_{\lambda,V}$ (which lie in the set $\mathcal{S}_{\varepsilon}$) may exist when $V$ is suboptimal. Thus, in order to obtain a larger collection of points in $\mathcal{S}_{\varepsilon}$, one should not fully minimize $\mathcal{P}_{\lambda,V}$. We propose two simple strategies to avoid this issue and produce a better approximation to $\mathcal{S}_{\varepsilon}$. The first is to stop the minimization routine prematurely by relaxing the tolerances. However, this prevents the computation of points where the functions used to define the stopping criteria take on values much smaller than their respective prescribed tolerances. Such points also lie in $\mathcal{S}_{\varepsilon}$. Therefore, one should repeat the computation for progressively tighter tolerance values. The second approach is to use tight tolerances and store the sequence of points generated by the minimization procedure as it progresses. {As explained above, many of these points are expected to lie in the set $\mathcal{S}_{\varepsilon}$, which can lead to larger sections of the extremal trajectory being approximated.} The two strategies are not exclusive, and the best results may be obtained when combining them. Whichever method is used, the process can be repeated by initializing the minimization algorithm from different random initial conditions $\vec{a}_0 \in \Omega$, in order to generate a large collection of points in $\mathcal{S}_{\varepsilon}$. Observe that this process is amenable to a large degree of parallelization, and can be scaled to high-dimensional systems by choosing inexpensive minimization routines such as the BFGS quasi-Newton algorithm~\cite{Broyden1970,Fletcher1970,Goldfarb1970,Shanno1970}. Our first approach was used in conjunction with the BFGS quasi-Newton algorithm to obtain the results presented in \cref{s:upos-shear-flow}. The MATLAB built-in function \texttt{fminunc} was used to implement the minimization, with the step tolerance on $\vec{a}$ (relative lower bound on the size of a step) set to $10^{-6}$ for all sets of results, and the first-order optimality tolerance (lower bound on $\abs{\abs{\nabla \mathcal{P}_{\lambda,V} }}_{\infty}$) set to $10^{-6}$ for all results except for those presented in figure \ref{fig:PO_Approx250}, where it was set to $5 \times 10^{-7}$.

Any points computed using this heuristic procedure provide educated initial guesses for algorithms that converge to UPOs by evolving the system's dynamics forward in time. However, it should be stressed that there are no theoretical guarantees that our heuristic procedure will produce sufficiently accurate approximations to the extremal orbit. First, it is known that there can exist points in $\mathcal{S}_\varepsilon$ that are not close to the extremal or near-extremal trajectories~\cite[Section 4]{Tobasco2018}. Second, the analysis in~\cite{Tobasco2018} does not guarantee that the entire extremal trajectory is in $\mathcal{S}_{\varepsilon}$, but only that it spends a large fraction of a finite period of time in $\mathcal{S}_{\varepsilon}$ when $V$ is close to optimal. Therefore, one may not be able to approximate the full extremal trajectory in practice. Nevertheless, the results presented in~\cite{Tobasco2018} and those in \cref{s:upos-shear-flow} demonstrate that very good approximations of the extremal orbits are often obtained in practice when $V$ is very close to optimal.

\section{Numerical optimization of auxiliary functions}
\label{s:sos-bounds}
The approach to approximating extremal UPOs described in the previous section relies on the availability of a near-optimal auxiliary function. The following subsections describe how suitable polynomial $V$ and $\lambda$ can be constructed computationally using SOS optimization when $\Phi$ and (the entries of) $\vec{f}$ are polynomials. Readers who are primarily interested in the application of our methods can safely skip this section at first reading and proceed to \cref{s:shear-flow-results} for numerical results.

\subsection{Optimizing auxiliary functions with SOS optimization}

In order to compute near-optimal auxiliary functions numerically, we begin by observing that, since $\Omega$ is compact by assumption, polynomials are dense in $C^1(\Omega)$~\cite[Ch. 1, Theorem 1.1.2]{Llavona1986}. Thus, one may restrict the search for auxiliary functions in~\cref{e:inf-sup-V} to the space $\Pi_n$ of $n$-variate polynomials:
\begin{equation}
\label{e:inf-sup-V-poly}
\overline{\Phi}^* = 
\adjustlimits \inf_{V \in \Pi_n} \max_{\vec{a} \in \Omega} \left\{ \Phi(\vec{a}) + \vec{f}(\vec{a})\cdot \nabla V(\vec{a}) \right\}.
\end{equation}
This seemingly small refinement of~\cref{e:inf-sup-V}, which is our first contribution, is actually crucial to prove that a near-optimal $V$ can be computed with the methods described in this section (see \cref{th:convergence-wsos} below). Of course, replacing $C^1$ functions with polynomials may prevent the infimum over $V$ in~\cref{e:inf-sup-V-poly} from being attained even when that in~\cref{e:inf-sup-V} is, but this will not be important for our purposes. Whether the infimum in~\cref{e:inf-sup-V} is attained and under which conditions an optimal polynomial $V$ exists are open theoretical questions that go beyond the scope of the present work.

The space of all $n$-variate polynomials is still infinite-dimensional, hence computationally intractable. To make progress, one can limit the search for $V$ to the set $\Pi_{n,d}$ of $n$-variate polynomials of degree $d$ or less, at the cost of replacing the equality in~\cref{e:inf-sup-V-poly} with an upper bound. For each integer $d$, the best upper bound on  $\overline{\Phi}^*$ available with $V \in \Pi_{n,d}$ is given by
\begin{align}
\label{e:inf-sup-V-poly-reformulated}
\overline{\Phi}^* 
&\leq \adjustlimits \inf_{V \in \Pi_{n,d}} \max_{\vec{a} \in \Omega} \left\{ \Phi(\vec{a}) + \vec{f}(\vec{a})\cdot \nabla V(\vec{a}) \right\} \\
&= \inf_{\subalign{V &\in \Pi_{n,d} \\ \lambda &\in \R}}
\left\{  \lambda \mid\, 
\lambda - \Phi(\vec{a}) - \vec{f}(\vec{a})\cdot \nabla V(\vec{a}) \geq 0 \text{ on }\Omega 
\right\}. \notag 
\end{align}

The right-hand side is a convex and finite-dimensional minimization problem with a polynomial inequality constraint. The optimization variables are $\lambda$ and, in the most general case, the $\binom{n+d}{d}$ coefficients of $V$. Numerical solution of this problem to global optimality is difficult because polynomial inequalities are NP--hard except for a few special cases, such as univariate or quadratic polynomials~\cite{Murty1987}. To reduce the computational complexity, a common strategy is to replace nonnegativity with the sufficient condition that $\lambda - \Phi(\vec{a}) - \vec{f}(\vec{a})\cdot \nabla V(\vec{a})$ is representable as a sum of squares of other polynomials of degree no larger than $r(d)/2$, where
\begin{align}
r(d) 
&= \deg\left( \lambda - \Phi - \vec{f}\cdot \nabla V \right) \\
&= \max\left\{ \deg(\Phi), \deg(\vec{f})+d-1 \right\}. \notag 
\end{align}
While not all nonnegative polynomials admit an SOS representation, its existence (or lack thereof) can be established in polynomial time by solving an SDP~\cite{Parrilo2003,Parrilo2012}. Additionally, a variety of open-source software packages are available to automatically reformulate optimization problems with SOS constraints as SDPs and solve them. Upon strengthening the nonnegativity constraint in~\cref{e:inf-sup-V-poly-reformulated} with an SOS constraint, one obtains the computable upper bound
\begin{equation}
\label{e:sos-bound}
\overline{\Phi}^* \leq \inf_{\subalign{V &\in \Pi_{n,d} \\ \lambda &\in \R}}
\left\{  \lambda \mid\,  \lambda - \Phi - \vec{f}\cdot \nabla V  \in \Sigma_{n,r(d)}
\right\},
\end{equation}
where $\Sigma_{n,q}$ denotes the set of $n$-variate SOS polynomials of degree no larger than $q$. By increasing $d$, the degree of $V$, one obtains a sequence of nonincreasing bounds on $\overline{\Phi}^*$. This is the approach followed in~\cite{Chernyshenko2014,Fantuzzi2016,Goluskin2018a,Goluskin2019}.

The SOS constraint in~\cref{e:sos-bound} enforces that $\lambda - \Phi(\vec{a}) - \vec{f}(\vec{a})\cdot \nabla V(\vec{a})$ is nonnegative everywhere on $\R^n$, not only on $\Omega$. A weaker sufficient condition for nonnegativity on $\Omega$, which also relies on SOS polynomials, can be formulated if $\Omega$ is a semialgebraic set. Precisely, assume that
\begin{equation}
\label{e:Omega-def}
\Omega = \left\{ \vec{a} \in \R^n \mid\; g_1(\vec{a})\geq 0,\, \ldots,\, g_m(\vec{a}) \geq 0 \right\}
\end{equation}
for some polynomials $g_1,\,\ldots,\,g_m$ of degree $s$ or less. For each integer $d$ such that $r(d)\geq s$, consider the set
\begin{equation}
\Lambda_d := \left\{ \sigma_0(\vec{a}) + \sum_{i=1}^m \sigma_i(\vec{a}) g_i(\vec{a}) \mid\, 
\sigma_0 \in \Sigma_{n,r(d)},\;
\sigma_1,\ldots,\,\sigma_m \in \Sigma_{n,r(d)-s}\right\}.
\end{equation}
In other words, polynomials in $\Lambda_d$ are weighted sums of SOS polynomials, where the weights are exactly the polynomials defining $\Omega$. Clearly, polynomials in $\Lambda_d$ are non-negative over $\Omega$ and $\Sigma_{n,r(d)} \subset \Lambda_d$. Since weighted SOS constraints can also be transformed into SDPs\footnote{A weighted SOS constraint $p \in \Lambda_d$ is equivalent to the $m+1$ SOS constraints $p - \sum_{i=1}^m \sigma_i g_i \in \Sigma_{n,r(d)}$ and $\sigma_1,\ldots,\sigma_m \in \Sigma_{n,r(d)-s}$. This formulation of weighted SOS constraints, also referred to as the generalized S procedure~\cite{Tan2006a,Fantuzzi2016}, is slightly suboptimal for computational purposes than the approach described in~\cite[Section 2.4]{Lasserre2015}, because it introduces more optimization variables than necessary. However, it is convenient because most software packages for SOS optimization do not natively support the method of~\cite[Section 2.4]{Lasserre2015}.}, a computationally tractable upper bound on $\overline{\Phi}^*$ that improves on~\cref{e:sos-bound} is
\begin{equation}
\label{e:wsos-problem}
\overline{\Phi}^* \leq \inf_{\subalign{V &\in \Pi_{n,d} \\ \lambda &\in \R}}
\left\{  \lambda \mid\,  \lambda - \Phi - \vec{f}\cdot \nabla V  \in \Lambda_d
\right\}.
\end{equation}

The real advantage of using weighted SOS constraints is that, if $\Omega$ satisfies an additional mild condition, then we can guarantee that the infimum on the right-hand side converges to $\overline{\Phi}^*$ as the degree $d$ of $V$ tends to infinity. This means that, at least in principle, arbitrarily sharp bounds on $\overline{\Phi}^*$ can be computed by optimizing polynomial auxiliary functions of sufficiently high degree using SOS optimization. \Cref{th:convergence-wsos} below formalizes these observations and is one of the main contributions of this paper. Its proof, which we report in detail below for completeness, is a standard argument in SOS optimization and uses a result due to Putinar~\cite[Lemma 4.1]{Putinar1993} on the existence of weighted SOS representations for strictly positive polynomials on a class of compact semialgebraic sets. For a detailed discussion of this result, known in the literature as Putinar's Positivstellensatz, we refer the reader to section 2.4 in~\cite{Lasserre2015}.
\begin{theorem}
	\label{th:convergence-wsos}
	Suppose that $\Omega$ is a compact semialgebraic set and that there exists $L$ such that $L - \|\vec{a}\|^2 \in \Lambda_d$ for some integer $d$. Then,
	\begin{equation}
	\overline{\Phi}^* = \lim_{d \to +\infty} \inf_{\subalign{V &\in \Pi_{n,d} \\ \lambda &\in \R}}
	\left\{  \lambda \mid\,  \lambda - \Phi - \vec{f}\cdot \nabla V  \in \Lambda_d
	\right\}.
	\end{equation}	
\end{theorem}
\begin{proof}
	We only need to show that for every $\varepsilon>0$ there exists an integer $d$ and a polynomial $V \in \Pi_{n,d}$ such that
	$\overline{\Phi}^* + \varepsilon - \Phi - \vec{f}\cdot \nabla V  \in \Lambda_d$. Equality~\cref{e:inf-sup-V-poly} guarantees the existence of an integer $d'$ and a polynomial $V \in \Pi_{n,d'}$ such that $\overline{\Phi}^* +  \varepsilon - \Phi - \vec{f}\cdot \nabla V$ is \textit{strictly} positive on $\Omega$. Then, by virtue of our assumptions on $\Omega$, Putinar's Positivstellensatz~\cite[Lemma 4.1]{Putinar1993} guarantees that $\overline{\Phi}^* + \varepsilon - \Phi - \vec{f}\cdot \nabla V$ belongs to the set of weighted SOS polynomials $\Lambda_{d''}$ for some integer $d''$. Setting $d=\max(d',d'')$ concludes the proof since $\Pi_{n,d'} \subset \Pi_{n,d}$ and $\Lambda_{d''} \subset \Lambda_d$.
\end{proof}
This result is dual to Theorem 2 in~\cite{Korda2018a} when the latter is used in the context of bounding time averages for ODEs. This relationship arises from the fact that polynomial auxiliary functions are dual to sequences of approximations to the moments of the invariant measure supported on the trajectory achieving the optimal average $\overline{\Phi}^*$; see~\cite{Lasserre2015} for more details. One could appeal to this duality and deduce \cref{th:convergence-wsos} from the analysis in~\cite{Korda2018a},
{but the direct proof we have given here is simpler and requires no knowledge of measure theory.} On the other hand, we must stress that the analysis of~\cite{Korda2018a} applies to a broad class of optimization problems over invariant measures, {which includes bounding not only infinite-time averages of ODEs, but also averages for discrete-time dynamical systems and stationary expectations of discrete- and continuous-time stochastic processes with compact state space.}

{The approach taken in our work can be used it these cases, too, if one proves that discrete-time averages and stochastic expectations admit exact characterizations in terms of auxiliary functions similar to~\eqref{e:inf-sup-V}. One way to do this is to apply convex duality to the measure-theoretic characterization of time averages and stochastic expectations given in~\cite{Korda2018a}, which is exactly the strategy followed independently by~\cite{Tobasco2018} to prove~\eqref{e:inf-sup-V}. However, alternative approaches that do not use measure theory also exist. For instance, a ``mollification" argument was used to construct auxiliary functions arbitrarily close to optimal when studying optimal control~\cite{Hernandez1996} and extreme events~\cite{Fantuzzi2019siads}. Such constructions may be easier than formulating measure-theoretic proofs when studying properties beyond time averages or stochastic expectations.}

Finally, observe that the assumption that $L - \|\vec{a}\|^2 \in \Lambda_d$ is mild and what is really needed is compactness. Indeed, any compact $\Omega$ is contained within some ball $L-\|\vec{a}\|^2 \geq 0$, so we can always ensure that $L - \|\vec{a}\|^2 \in \Lambda_d$ by adding $L-\|\vec{a}\|^2 \geq 0$ to the list of polynomial inequalities that define $\Omega$.

\subsection{Exploiting symmetry in weighted SOS constraints}
\label{ss:exploiting-symmetry}
The computational cost of solving SOS optimization problems increases quickly with problem size. As a result, solving the optimization problem on the right hand side of~\cref{e:wsos-problem} is practical only for small or medium sized problems, such that $\binom{n+d}{d}$ is $\mathcal{O}(100)$ or so. At the time of writing, the most computationally demanding part of our SOS approach to approximating extremal trajectories is the computation of near-optimal $V$ and $\lambda$ using SOS optimization. This is due to the poor scalability of general-purpose algorithms for SDPs. To reduce computational cost, one can exploit information about the boundedness and symmetry of $V$ that can be deduced a priori~\cite[Appendix~A]{Goluskin2019}. \Cref{th:wsos-multipliers} below shows that symmetry can be exploited in weighted SOS constraints, too. A proof of this theorem is also presented below, where we adapt the argument of~\cite[Proposition 1]{Goluskin2019}.
\begin{theorem}
	\label{th:wsos-multipliers}
	Let $\mathcal{T}:\mathbb{R}^n \to \mathbb{R}^n$ be an invertible linear transformation that generates a finite symmetry group, meaning that $\mathcal{T}^K$ is the identity for some integer $K$. Suppose that the system, the function $\Phi$ and the set $\Omega$ are invariant under $\mathcal{T}$, meaning that $\vec{f}(\mathcal{T}\vec{a})=\mathcal{T}\vec{f}(\vec{a})$, $\Phi(\mathcal{T}\vec{a})=\Phi(\vec{a})$ and $g_i(\mathcal{T}\vec{a})=g_i(\vec{a})$ for all $i \in \{1,\dots,m\}$. If there exist $V, \sigma_0, \sigma_1, \dots, \sigma_m$ that prove a bound $\lambda$ on $\overline{\Phi}^*$ via the sufficient condition in~\cref{e:wsos-problem}, then there exist $\widehat{V}, \widehat{\sigma}_0, \widehat{\sigma}_1, \dots, \widehat{\sigma}_m$ that are $\mathcal{T}$-- invariant and prove the same bound.
\end{theorem}
\begin{proof}
	Suppose that there exist $V,\sigma_0,\sigma_1,\dots,\sigma_m$ which satisfy the weighted SOS constraint in \cref{e:wsos-problem}:
	\begin{equation}
	\label{e:wsos-multipliers-proof:1}
	\lambda - \Phi\left(\vec{a}\right) - \vec{f}\left(\vec{a}\right) \cdot \nabla V\left(\vec{a}\right) = \sigma_0\left(\vec{a}\right) + \sigma_1\left(\vec{a}\right)g_1\left(\vec{a}\right) + \dots + \sigma_m\left(\vec{a}\right)g_m\left(\vec{a}\right).
	\end{equation}
	Consider now symmetrized versions of $V,\sigma_0,\sigma_1,\dots,\sigma_m$:
	\begin{subequations} \label{e:wsos-multipliers-proof:2}
		\begin{align}
		\widehat{V}\left(\vec{a}\right) &:= \frac{1}{K} \sum_{k=0}^{K-1} V(\mathcal{T}^k\vec{a}),  \\
		\widehat{\sigma}_i\left(\vec{a}\right) &:= \frac{1}{K} \sum_{k=0}^{K-1} \sigma_i(\mathcal{T}^k\vec{a}) \quad \mathrm{for} \quad i = 0,1,\dots,m.
		\end{align}
	\end{subequations}
	Note that $\widehat{V}(\mathcal{T}\vec{a}) = \widehat{V}(\vec{a})$ since $\mathcal{T}^K$ is the identity, and similarly for $\widehat{\sigma}_0, \widehat{\sigma}_1,\dots,\widehat{\sigma}_m$. The claim is proven if:
	\begin{equation}
	\label{e:wsos-multipliers-proof:3}
	\lambda - \Phi(\vec{a}) - \vec{f}(\vec{a}) \cdot \left[\frac{1}{K}\sum_{k=0}^{K-1} (\mathcal{T}^k)^T \nabla V(\mathcal{T}^k \vec{a}) \right] = \widehat{\sigma}_0(\vec{a}) + \widehat{\sigma}_1(\vec{a})g_1(\vec{a}) + \dots + \widehat{\sigma}_m(\vec{a})g_m(\vec{a}).
	\end{equation}
	To show that the above equality holds, evaluate~\cref{e:wsos-multipliers-proof:1} at $\mathcal{T}^k\vec{a}$ and use the symmetries of $\vec{f}$ and $\Phi$ to obtain the following:
	\begin{equation}
	\label{e:wsos-multipliers-proof:4}
	\lambda - \Phi({\vec{a}}) - \vec{f}(\vec{a}) \cdot [(\mathcal{T}^k)^T \nabla V(\mathcal{T}^k \vec{a} )] = \sigma_0(\mathcal{T}^k\vec{a}) + {\sigma}_1(\mathcal{T}^k\vec{a})g_1(\vec{a}) + \dots + {\sigma}_m(\mathcal{T}^k\vec{a})g_m(\vec{a}).
	\end{equation}
	Averaging both sides of \cref{e:wsos-multipliers-proof:4} for $k = 0,1,\dots,K-1$ gives \cref{e:wsos-multipliers-proof:3}, thus proving the claim.	
\end{proof}

A similar argument also holds if one wishes to prove that a set is globally absorbing using an SOS relaxation of the sufficient condition \cref{e:absorbing-set-conditions}. Specifically, if the ODE is invariant under a linear transformation $\mathcal{T}$, then there is no loss of generality in assuming that the function $W$ in an SOS relaxation of the inequality \cref{e:absorbing-set-conditions} is also $\mathcal{T}$-- invariant.

\section{Application to a model of shear flow}
\label{s:shear-flow-results}
To demonstrate the techniques described so far, we apply them to study a nine-dimensional quadratic ODE system. The system was introduced by Moehlis \textit{et al.}~\cite{Moehlis2004} as a model of sinusoidally forced shear flow in a periodic channel, with periods $L_x$ and $L_z$ in the streamwise and spanwise directions, respectively. We do not aim to analyze this model in detail, but only to showcase our new method of finding periodic orbits. The system takes the form
\begin{equation}
\label{e:nine-mode-system}
\frac{da_i}{dt} = \frac{1}{\Rey} \lambda_1 \delta_{1i} - \frac{1}{\Rey} \lambda_{i}a_i + N_{ijk} a_j a_k, \qquad i,j,k = 1,\dots,9,
\end{equation}
where summation over indices $j$ and $k$ is assumed, and $\delta_{1i}$ is the usual Kronecker delta. The state $\vec{a} = (a_1,\dots,a_9)$ represents the amplitude of physically relevant flow modes, $\Rey$ is the Reynolds number, and $\lambda_{i}$, $N_{ijk}$ are numerical coefficients. All coefficients $\lambda_{i}$ are strictly positive, so all modes are linearly damped, and $\lambda_1 \leq \lambda_i \leq \lambda_9$. The coefficients $N_{ijk}$, instead, are such that $N_{ijk}a_ia_ja_k = 0$, meaning that the quadratic terms conserve energy. Here we consider numerical values corresponding to the ``NBC" configuration in~\cite{Moehlis2004,Moehlis2005}, for which $L_x = 4\pi$ and $L_z = 2\pi$. For this as well as all other possible configurations, solutions of~\cref{e:nine-mode-system} exhibit two symmetries:
\begin{subequations}
\label{e:symmetries}
	\begin{align}
	\mathcal{T}_1 \vec{a} &= \left(a_1, a_2, a_3, -a_4, -a_5, -a_6, -a_7, -a_8, a_9\right), \\
	\mathcal{T}_2 \vec{a} &= \left(a_1, -a_2, -a_3, a_4, a_5, -a_6, -a_7, -a_8, a_9\right). 
	\end{align}
\end{subequations}
The symmetries~(\ref{e:symmetries}a,b) represent invariance of original flow field modeled by \cref{e:nine-mode-system} under translations of half a period in the streamwise and spanwise directions, respectively, and generate a four-element group $\{1, \mathcal{T}_1, \mathcal{T}_2, \mathcal{T}_1\mathcal{T}_2\}$. As noted in~\cite{Moehlis2005}, if one finds a periodic or fixed point solution to~\eqref{e:nine-mode-system}, there will be up to three other symmetry-related periodic orbits or fixed points obtained by the actions of the elements of the symmetry group.

At all values of $\Rey$, system~\cref{e:nine-mode-system} has a locally stable equilibrium point $\vec{a}_l = (1,0,\dots,0)$, which represents the laminar flow state (note that $N_{111} = 0$). For our chosen configuration, unsteady and often chaotic solutions corresponding to turbulent flows are observed for $\Rey \geq 80.54$, and a large collection of UPOs have been computed~\cite{Moehlis2005}. The laminar state is expected to be globally asymptotically stable below this value of $\Rey$, and Lyapunov functions certifying global stability were constructed using SOS optimization in~\cite{Goulart2012} for $\Rey < 54.1$. For larger $\Rey$, Chernyshenko \textit{et al.}~\cite{Chernyshenko2014} optimized lower bounds on the infinite-time average of the energy dissipation rate,
\begin{equation}
\mathcal{D} := \frac{\sum_{i = 1}^9 \lambda_{i} a_i^2}{\Rey},
\end{equation}
using polynomial auxiliary functions of degree up to 8. Here we repeat the computations using polynomials up to degree 10. We also compute bounds on the infinite-time average of the energy of perturbations from $\vec{a}_l$ (henceforth referred to as perturbation energy),
\begin{equation}
\mathcal{E} := (1 - a_1)^2 + \sum_{i=2}^9 a_i^2.
\end{equation}
In both cases, we use the auxiliary functions obtained with SOS optimization to approximate the extremal trajectories as described in \cref{s:po-approximation}. All our SOS computations were implemented using the MATLAB modeling toolbox YALMIP~\cite{Lofberg2004} and the SDP solver MOSEK~\cite{MOSEKApS2018}. We also exploited the symmetries defined by (\ref{e:symmetries}a,b) as described in \cref{ss:exploiting-symmetry} in order to reduce computational cost.

It can be shown that the unit ball centered at the origin absorbs trajectories of~\cref{e:nine-mode-system}. \Cref{th:convergence-wsos} thus guarantees the existence of near-optimal $V$ for the weighted SOS problem~\cref{e:wsos-problem}. Although the existence of near-optimal $V$ is not known to be guaranteed for the standard SOS problem~\cref{e:sos-bound}, preliminary investigations suggested that it also yields near-optimal $V$ and good orbit approximations. All numerical results reported below were therefore obtained by solving~\cref{e:sos-bound}, because it is computationally less expensive than~\cref{e:wsos-problem}.

\subsection{Bounds on the infinite-time average of energy dissipation rate}
\label{s:bounds-dissipation}
\begin{figure}[t]
	\centering
	\includegraphics[width=0.98\textwidth]{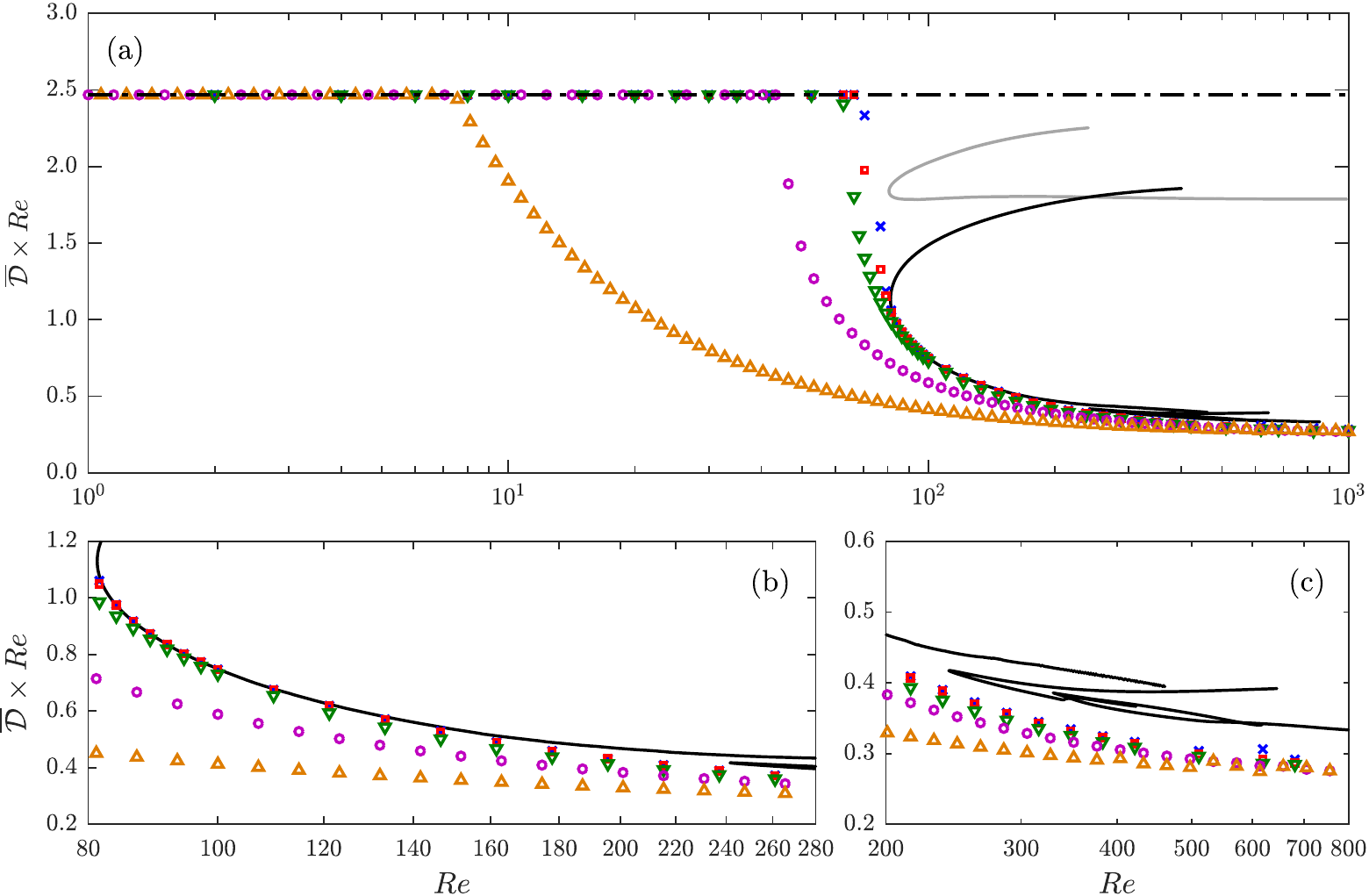}
	\caption{(a) Lower bounds on $\overline{\mathcal{D}}$ computed by optimizing over polynomial auxiliary functions of degree 2 (\mytriangle{matlaborange}), 4 (\mycircle{matlabmagenta}), 6 (\myudtriangle{matlabgreen}), 8 (\mysquare{matlabred}) and 10 (\mycross{matlabblue}). Also shown is a family of UPOs computed by Moehlis \textit{et al.}~\protect\cite{Moehlis2005} ({\color{matlabgrey}\solidrule}), and three new families of UPOs ({\color{black}\solidrule}). 
		(b,c) Detailed views of the results in panel (a). Note the gap between the degree-10 bounds and the currently available UPO data at $\Rey \gtrapprox 125$.
	}
	\label{fig:EDR_Bounds}
\end{figure}
\cref{fig:EDR_Bounds} illustrates lower bounds on $\overline{\mathcal{D}}$ computed by solving~\cref{e:sos-bound} with $\Phi = -\mathcal{D}$ for $d = 2,4,6,8$ and 10. The polynomial auxiliary functions $V$ used here are of more general ans\"atze, and some of higher degree, than those used to derive the bounds in~\cite{Chernyshenko2014}, which results in tighter bounds. Also shown in the figure are one of the families of UPOs found in~\cite{Moehlis2005}, born at $\Rey = 80.54$, and three branches discovered using the strategy outlined in \cref{s:po-approximation} (see \cref{s:upos-shear-flow} for more details). These branches, which we refer to as PO1, PO2 and PO3, are born at $\Rey = 81.24$, $\Rey = 241.5$ and $\Rey = 330.2$ respectively. To the best of our knowledge, they are new and not among those reported in~\cite{Moehlis2005}.

In the range $82 \lessapprox \Rey \lessapprox 125$, as the degree of $V$ increases, the lower bounds converge to the lower PO1 branch, which therefore represents a family of minimal orbits for $\overline{\mathcal{D}}$ for this range of Reynolds numbers. At $\Rey = 89$, for example, our numerically computed bound with degree-10 $V$ is only $0.05\%$ less than the average over the numerically computed periodic orbit. This confirms the hypothesis that the gap between the bounds on $\overline{\mathcal{D}}$ in~\cite[Figure 1]{Chernyshenko2014} and values corresponding to the laminar solution $\vec{a}_l$, as well as that between the bounds and the simulated dissipation rate of the turbulent flow, was due to the existence of a branch of UPOs that saturates the bounds (at least in the aforementioned range of $\Rey$).

From $\Rey \approx 125$, the SOS bounds begin to deviate from the lowermost branch of PO1 orbits. This can be seen clearly in panel (b) of \cref{fig:EDR_Bounds}. Increasing the degree of $V$ to 12 brings no significant improvement in the bounds (not shown in the Figure) for $\Rey > 125$, suggesting that the degree-10 $V$ used are almost optimal. We conjecture that there exists another family of UPOs with lower mean energy dissipation rate than all known UPOs, although attempts to find it have so far been unsuccessful. We will return to this issue in \cref{s:upos-shear-flow}.

Finally, for $\Rey \lessapprox 67$, the SOS bounds computed with degree-10 $V$ deviate from the values corresponding to the laminar state by less than the numerical tolerance of the SDP solver. We conjecture that the laminar solution is globally stable up to $\Rey = 80.54$---the point at which the first family of orbits is born---meaning that our SOS bounds in the range $67 \lessapprox \Rey \lessapprox 80.54$ are far from sharp and $V$ of very high degree is required to produce sharp bounds near criticality. We leave it to future work to determine possible reasons for this behavior, and whether it is generic.

\subsection{Bounds on the infinite-time average of perturbation energy}
\label{s:bounds-PE}
\begin{figure}[t]
	\centering
	\includegraphics[width=0.98\textwidth]{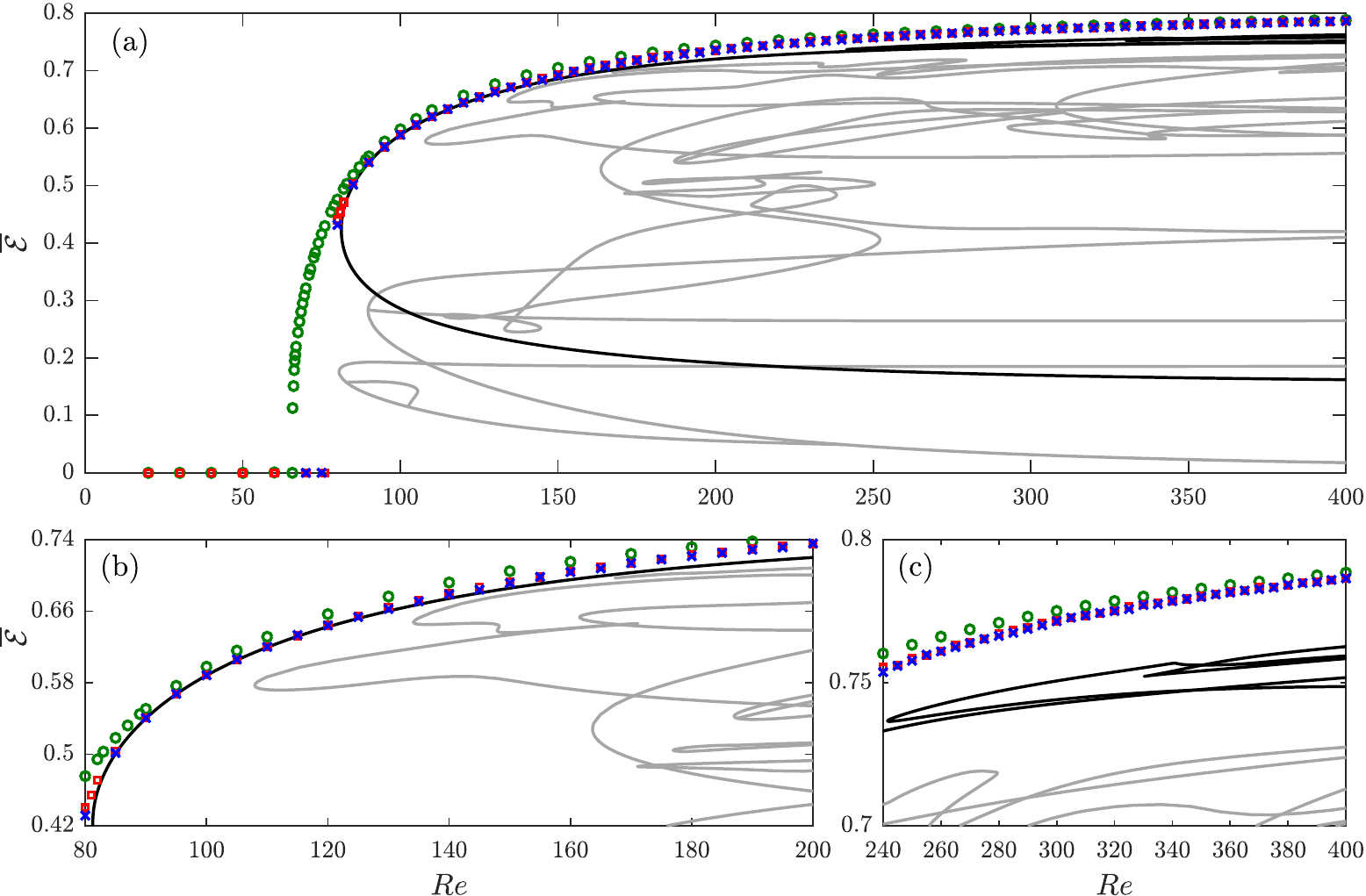}
	\caption{(a) Upper bounds on $\overline{\mathcal{E}}$ computed by optimizing over polynomial auxiliary functions of degree 6 (\mycircle{matlabgreen}), 8 (\mysquare{matlabred}) and 10 (\mycross{matlabblue}). Also shown are families of UPOs computed by Moehlis \textit{et al.}~\protect\cite{Moehlis2005} ({\color{matlabgrey}\solidrule}), and three new families of UPOs ({\color{black}\solidrule}). 
		(b,c) Detailed views of the results in panel (a).
	}
	\label{fig:PE_Bounds}
\end{figure}
\cref{fig:PE_Bounds} shows upper bounds on $\overline{\mathcal{E}}$ computed by solving~\cref{e:sos-bound} with $\Phi = \mathcal{E}$ for $d = 6, 8$ and 10. The figure also illustrates the mean perturbation energy of numerous families of UPOs discovered by Moehlis \textit{et al.}~\cite{Moehlis2005}, as well as that of the three new families PO1, PO2 and PO3. As the degree of $V$ is raised, in the range $82 \lessapprox \Rey \lessapprox 125$ the upper bounds on $\overline{\mathcal{E}}$ converge to the same branch of UPOs in the family PO1 that minimizes $\overline{\mathcal{D}}$. At $\Rey = 90$, for example, the bound computed with degree-10 $V$ is larger than the average over the UPO by only 0.08\%. This branch of orbits is simultaneously maximal for $\overline{\mathcal{E}}$ and minimal for $\overline{\mathcal{D}}$. As for the case of the mean energy dissipation rate, the SOS bounds on $\overline{\mathcal{E}}$ deviate from this branch for $\Rey \gtrapprox 125$. This can be seen clearly in panel (b) of \cref{fig:PE_Bounds}. Furthermore, very little improvement in the bound is observed as the degree of $V$ is increased beyond 10. Again, this suggests that the degree-10 $V$ used are almost optimal and that another, yet undiscovered, branch of UPOs becomes maximal for $\overline{\mathcal{E}}$. Based on the results at low $\Rey$ and the observations that solutions with lower dissipation rate have more available energy, we conjecture that this branch of UPOs will simultaneously maximize $\overline{\mathcal{E}}$ and minimize $\overline{\mathcal{D}}$. 

The SOS bounds with $V$ of degree 8 and 10 deviate from zero by less than the numerical tolerance of the SDP solver for $\Rey \lessapprox 76$. Since only the laminar solution $\vec{a}_l$ has a perturbation energy of zero, this suggests that it is globally stable in this range of $\Rey$. This statement could be made rigorous either by proving a zero upper bound on $\overline{\mathcal{E}}$ analytically, or by combining SOS computations and interval arithmetic to construct more general Lyapunov functions than those considered in~\cite{Goulart2012}. The details of such an approach are beyond the present discussion, but we refer the reader to~\cite{Goluskin2018a} for an example of solving SOS problems with interval arithmetic.

\subsection{Approximations to extremal periodic orbits}
\label{s:upos-shear-flow}
The same auxiliary functions optimized with SOS programming to bound $\overline{\mathcal{E}}$ and $\overline{\mathcal{D}}$ can be used to approximate the corresponding extremal trajectories as described in \cref{s:po-approximation}. Fixing $\Phi$ to be either $-\mathcal{D}$ or $\mathcal{E}$, problem~\cref{e:sos-bound} was solved by optimizing over polynomial auxiliary functions $V$ of the form $P_9(\vec{a}) + \| \vec{a} \|^{10}$ or $P_{10}(\vec{a})$, where $P_k(\vec{a}) \in \Pi_{9,k}$. In each case, the polynomial $\mathcal{P}_{\lambda,V}$ defined as in~\cref{e:polynomial} was built with numerically determined coefficients. We then proceeded by minimizing $\mathcal{P}_{\lambda,V}$ over $\mathbb{R}^9$ from random initial conditions uniformly distributed in the range $\Pi_{i = 1}^9 (a_i, b_i)$, where $a_i$ and $b_i$ are certain particular values in the interval $(-0.5, 0.5)$, which are not reported for brevity. The BFGS quasi-Newton algorithm~\cite{Broyden1970,Fletcher1970,Goldfarb1970,Shanno1970} was used for the minimization routine, to obtain points that lie in the set $\mathcal{S}_{\varepsilon}$ for some small $\varepsilon$. If the extremal trajectory is periodic, then it is contained in this set for a fraction of its time period no less than $1-\delta/\varepsilon$, where $\delta$ is the difference between the upper bound on $\overline{\Phi}^*$ we computed and the value of $\overline{\Phi}^*$ itself. The latter is not generally known, so one cannot compute $\delta$.\footnote{One could estimate $\delta$ from above using the difference between the bound and the average of $\Phi$ along any trajectory that is known a priori. However, this estimate is useful only if such a trajectory is nearly extremal, which is not the case in this paper.} However, the periodic orbit approximation method presented in \cref{s:po-approximation} does not require one to know this quantity.
\begin{figure}[t]
	\centering
	\includegraphics[width=0.98\textwidth]{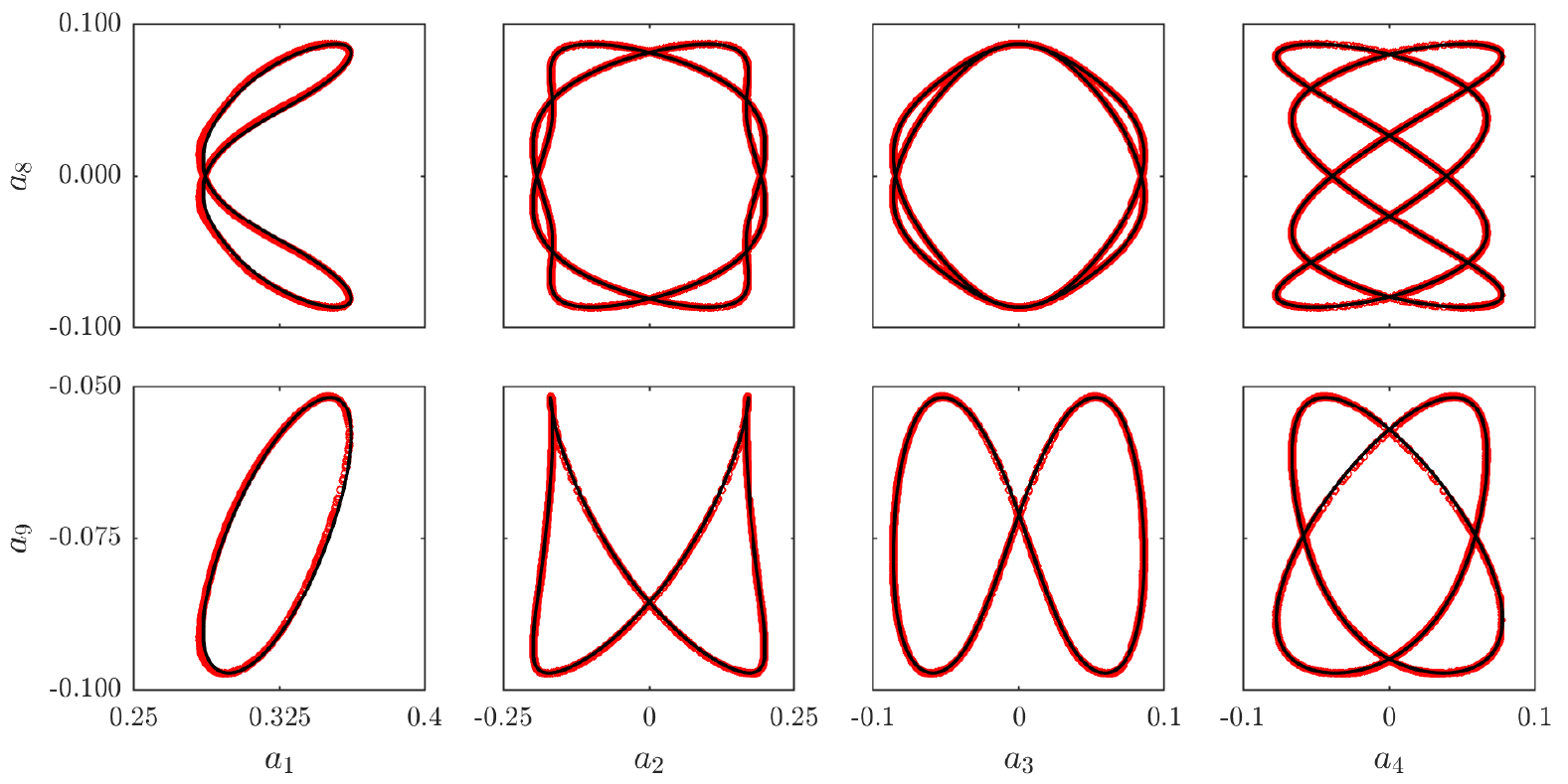}
	\caption{UPOs that maximize $\overline{\mathcal{E}}$ at $\Rey = 95$, projected onto various 2D subspaces ({\color{black}\solidrule}). Also shown are points in the set $\mathcal{S}_{2.55\mathrm{e}{-4}}$ (\mycircle{matlabred}), with the auxiliary function used of the form $P_9(\vec{a}) + \|\vec{a}\|^{10}$.} 
	\label{fig:PO_Approx95}
\end{figure}
\begin{figure}[h!]
	\centering
	\includegraphics[width=0.98\textwidth]{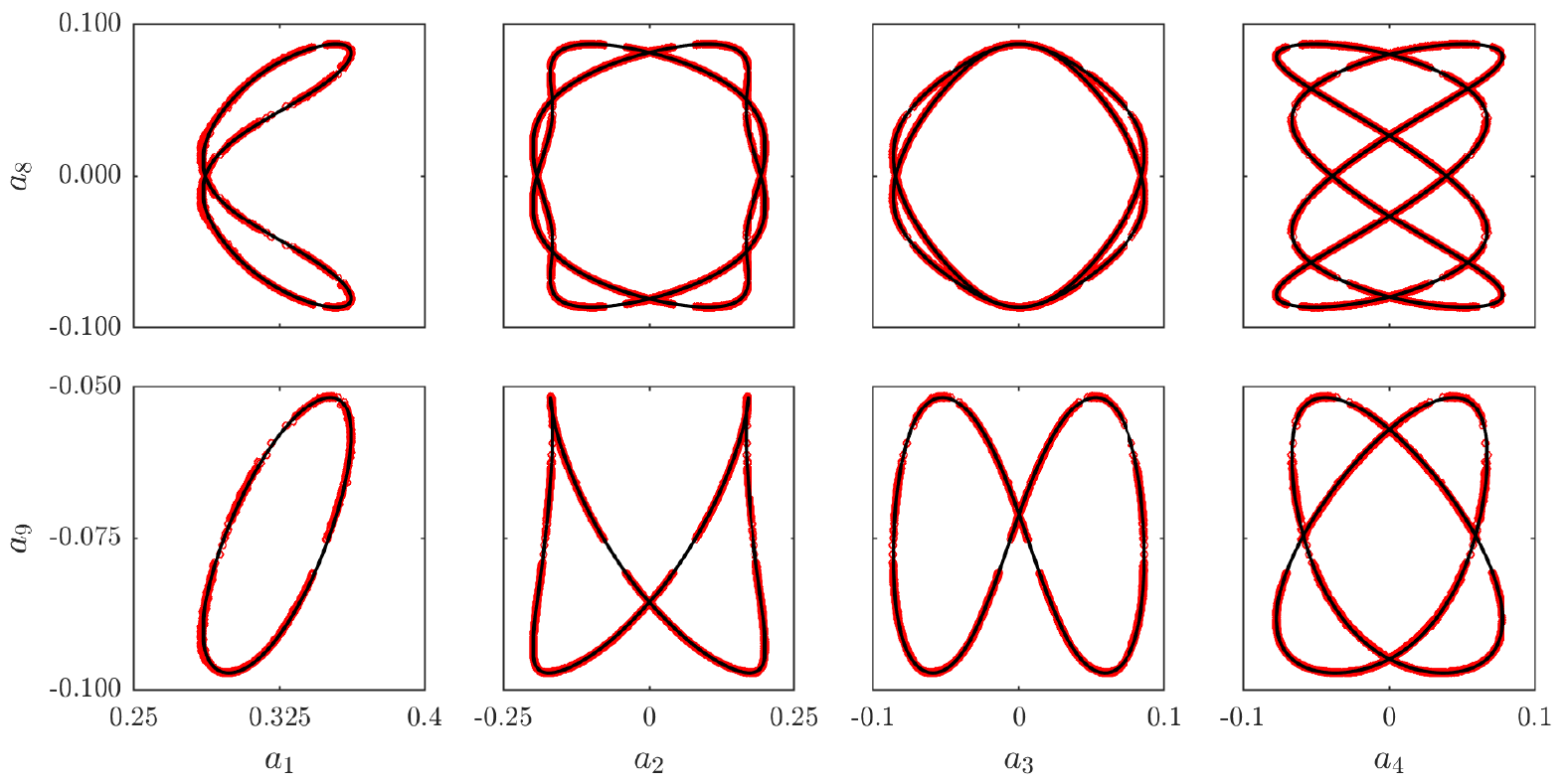}
	\caption{UPOs that minimize $\overline{\mathcal{D}}$ at $\Rey = 95$, projected onto various 2D subspaces ({\color{black}\solidrule}). Also shown are points in the set $\mathcal{S}_{1.23\mathrm{e}{-7}}$ (\mycircle{matlabred}), with the auxiliary function used of the form $P_9(\vec{a}) + \|\vec{a}\|^{10}$.} 
	\label{fig:PO_ApproxEDR95}
\end{figure}

The red dots in \cref{fig:PO_Approx95} show the SOS approximation to the trajectory maximizing $\overline{\mathcal{E}}$ at $\Rey = 95$. The results were obtained with $V = P_9(\vec{a}) + \|\vec{a}\|^{10}$ and $\varepsilon = 2.55 \times 10^{-4}$. These points resemble closed curves, suggesting that the extremal trajectory is a UPO. Indeed, using any of the points as the initial guess for a basic single-shooting Newton--Raphson algorithm~\cite[Ch. 7]{Cvitanovic2017},~\cite{Kelley2003}, we found a UPO that fits our approximation extremely well. Other UPOs at $\Rey = 95$ that are simultaneously maximal for $\overline{\mathcal{E}}$ can be obtained from symmetry considerations, as described after~(\ref{e:symmetries}b). All such UPOs are plotted as black lines in \cref{fig:PO_Approx95}. Numerical continuation of one of these symmetry-related orbits using the package \textsc{MatCont}~\cite{Dhooge2008} yields the family of UPOs so far referred to in this work as PO1. No bifurcations were detected using \textsc{MatCont} on the extremal PO1 branch (upper PO1 branch in \cref{fig:PE_Bounds}). The PO1 family is not among the UPOs reported in~\cite{Moehlis2005}, exists for $\Rey = 81.24$ and, as described in \cref{s:bounds-PE}, it maximizes $\overline{\mathcal{E}}$ up to numerical tolerance for $\Rey \lessapprox 125$.

The same UPOs shown in \cref{fig:PO_Approx95} are also plotted in \cref{fig:PO_ApproxEDR95}, this time alongside our SOS approximation of the minimal orbits for $\overline{\mathcal{D}}$. This approximation was obtained with $V=P_9(\vec{a}) + \|\vec{a}\|^{10}$ and $\varepsilon = 1.23 \times 10^{-7}$; a different value for $\varepsilon$ is used because near-optimal bounds on $\overline{\mathcal{D}}$ require a different $V$ than near-optimal bounds on $\overline{\mathcal{E}}$. It is clear that the same branch of PO1 orbits simultaneously minimizes the mean dissipation rate and maximizes the mean perturbation energy. Although the approximation obtained when bounding the former is slightly worse, our Newton--Raphson solver converged robustly to the same UPO using initial guesses from either approximation.

Surprisingly, for both choices of $\Phi$ specifying a more general ansatz for $V$ results in the same bound on $\overline{\Phi}^*$ (up to small differences due to working in finite precision), but a worse approximation of the extremal periodic orbits. \Cref{fig:PO_Approx95_P10} shows our SOS approximation of the orbits maximizing $\overline{\mathcal{E}}$, obtained with a generic degree-10 $V$ and $\varepsilon = 2.55 \times 10^{-4}$. The approximation is qualitatively correct but there is a marked offset of $\mathcal{O}(0.01)$ in the $a_1$ direction. Equally worse approximations are obtained with generic degree-10 $V$ when approximating the orbits minimizing $\overline{\mathcal{D}}$ (not shown for brevity). One possible explanation for this behavior is that multiple choices of $V$ yielding the same optimal or near-optimal bound on $\overline{\Phi}^*$ exist, but not all provide a good approximation to the extremal trajectories. Since the SOS problems~\cref{e:sos-bound} and~\cref{e:wsos-problem} for $V$ only optimize the bound on $\overline{\Phi}^*$, but do not take the quality of the extremal trajectory approximation into account, the numerical algorithm used to solve the SOS problem may converge to a $V$ that gives the same near-optimal bound on $\overline{\Phi}^*$ as many others (within reasonable numerical tolerances), but a poorer approximation of the extremal trajectory.
Another factor is that the optimization problem~\cref{e:sos-bound} becomes increasingly ill-conditioned as the degree of $V$ and the number of optimization variables increases. This is a common problem for large SOS programs~\cite{Lofberg2009}. Restricting the form of $V$ used in our computation counteracts both these facts, and we believe that this is why it results in a slightly better orbit approximation. A more systematic investigation, both theoretical and computational, is left for future work.
\begin{figure}
	\centering
	\includegraphics[width=0.98\textwidth]{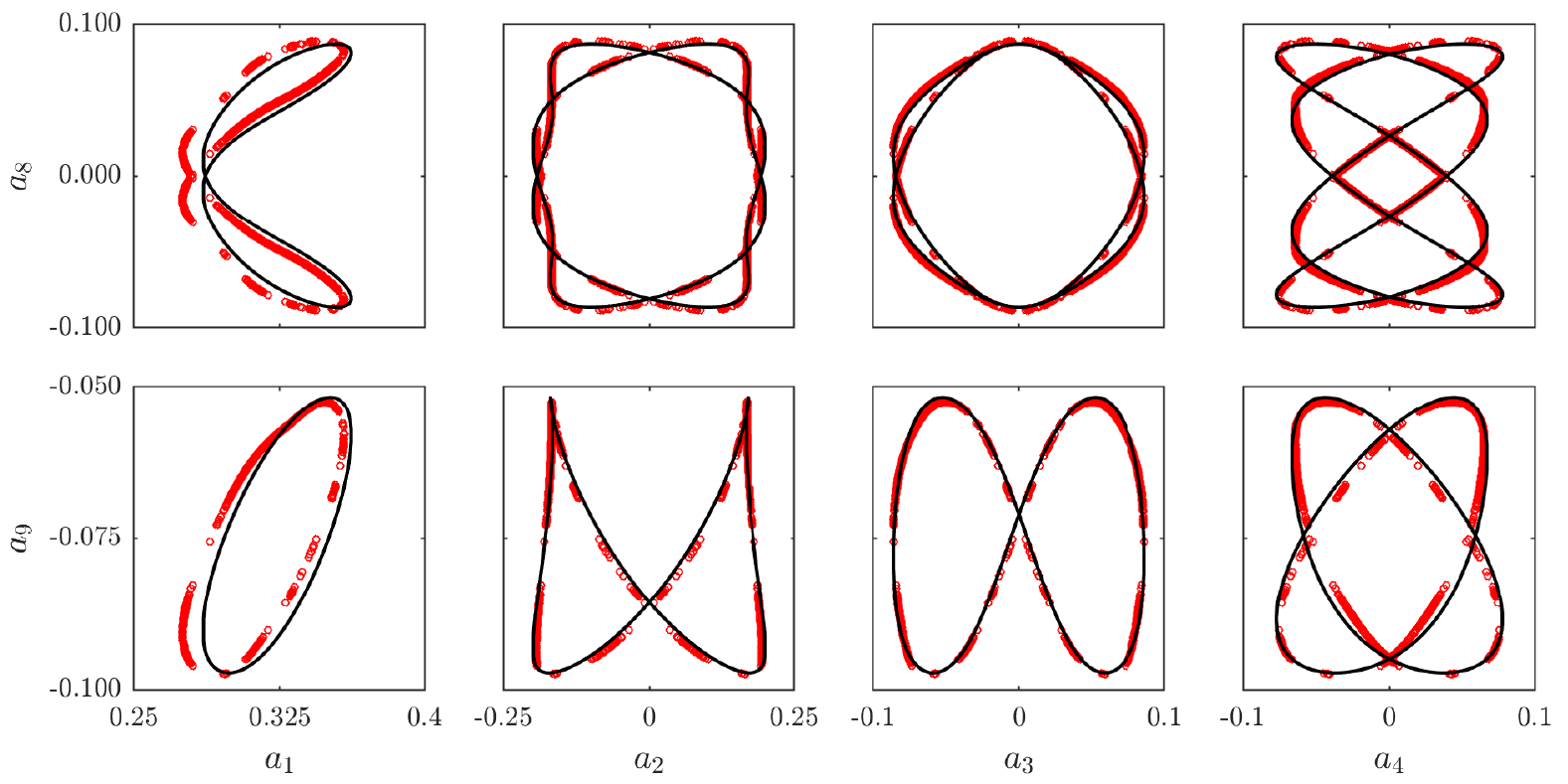}
	\caption{UPOs that maximize $\overline{\mathcal{E}}$ at $\Rey = 95$, projected onto various 2D subspaces ({\color{black}\solidrule}). Also shown are points in the set $\mathcal{S}_{2.55\mathrm{e}{-4}}$ (\mycircle{matlabred}), with the auxiliary function used of the form $P_{10}(\vec{a})$.} 
	\label{fig:PO_Approx95_P10}
\end{figure}

An identical approximation of the maximal orbits for $\overline{\mathcal{E}}$ was performed at $\Rey = 105$. At this higher value of $\Rey$, the SOS bounds on $\overline{\mathcal{E}}$ and $\overline{\mathcal{D}}$ still agree very well with the averages over the most extremal UPOs of the PO1 family. \Cref{fig:PO_Approx105} shows these orbits, together with the SOS approximations obtained after optimizing $V = P_9(\vec{a}) + \|\vec{a}\|^{10}$ to yield near-optimal bounds on $\overline{\mathcal{E}}$. The approximation points were computed with $\varepsilon = 3.93 \times 10^{-6}$. The results are just as convincing as at $\Rey = 95$, and equally good approximations are obtained when maximizing lower bounds on $\overline{\mathcal{D}}$.

Finally, we repeated the analysis at $\Rey = 250$, for which there remains a gap between the available UPO data and our bounds on both $\overline{\mathcal{E}}$ and $\overline{\mathcal{D}}$. \Cref{fig:PO_Approx250} compares the most extremal known UPOs at $\Rey = 250$ to our SOS approximations of the extremal orbits, which consist of points in the set $\mathcal{S}_{4.71\mathrm{e}{-6}}$ obtained after optimizing upper bounds on $\overline{\mathcal{E}}$ with $V=P_9(\vec{a}) + \|\vec{a}\|^{10}$. The two do not match and the SOS approximations do not look like closed orbits, although they do resemble sections of UPOs. Even worse results were obtained when maximizing lower bounds on $\overline{\mathcal{D}}$. It should be stressed once again that the approximation procedure of \cref{s:po-approximation} is not guaranteed to localize the entire extremal UPO, but only regions of state space where it spends a large fraction of its time period, provided that $V$ is very close to optimal. The small difference in the optimal bound on $\overline{\mathcal{E}}$ obtained with degree-8 and degree-10 $V$ (cf. \cref{fig:PE_Bounds}) suggests that the latter is close to optimal, and we conjecture that an extremal UPO at $\Rey = 250$ indeed passes near the approximating points shown in \cref{fig:PO_Approx250}. We also expect that $V$ of higher degree would produce a better approximation, but we could not test this due to both the increase in required computational resources and the aforementioned degradation in numerical conditioning.

Unfortunately, the approximation points in \cref{fig:PO_Approx250} did not provide sufficiently good initial conditions to converge to an extremal UPO at $\Rey = 250$ with our single-shooting Newton--Raphson algorithm. This is due in part to our degree-10 $V$ not being sufficiently close to optimal, but also due to the UPOs of the system becoming more and more unstable as the Reynolds number is raised. We believe that the latter issue is particularly relevant, and that it could be resolved with multiple-shooting or other, more sophisticated methods for converging UPOs, which perhaps take into consideration the characteristics of the set $\mathcal{S}_{\varepsilon}$.

For some initial points, however, our basic single-shooting Newton--Raphson algorithm did converge to a UPO not reported in~\cite{Moehlis2005}, which is plotted in \cref{fig:PO_Approx250} together with three others obtained from symmetry considerations. This UPO was then numerically continued using \textsc{MatCont} to produce the family previously referred to as PO2. The orbits shown in \cref{fig:PO_Approx250} belong to the upper branch of this family in \cref{fig:PE_Bounds}. A third family of UPOs, earlier referred to as PO3, bifurcates from this outer branch.
\begin{figure}
	\centering
	\includegraphics[width=0.98\textwidth]{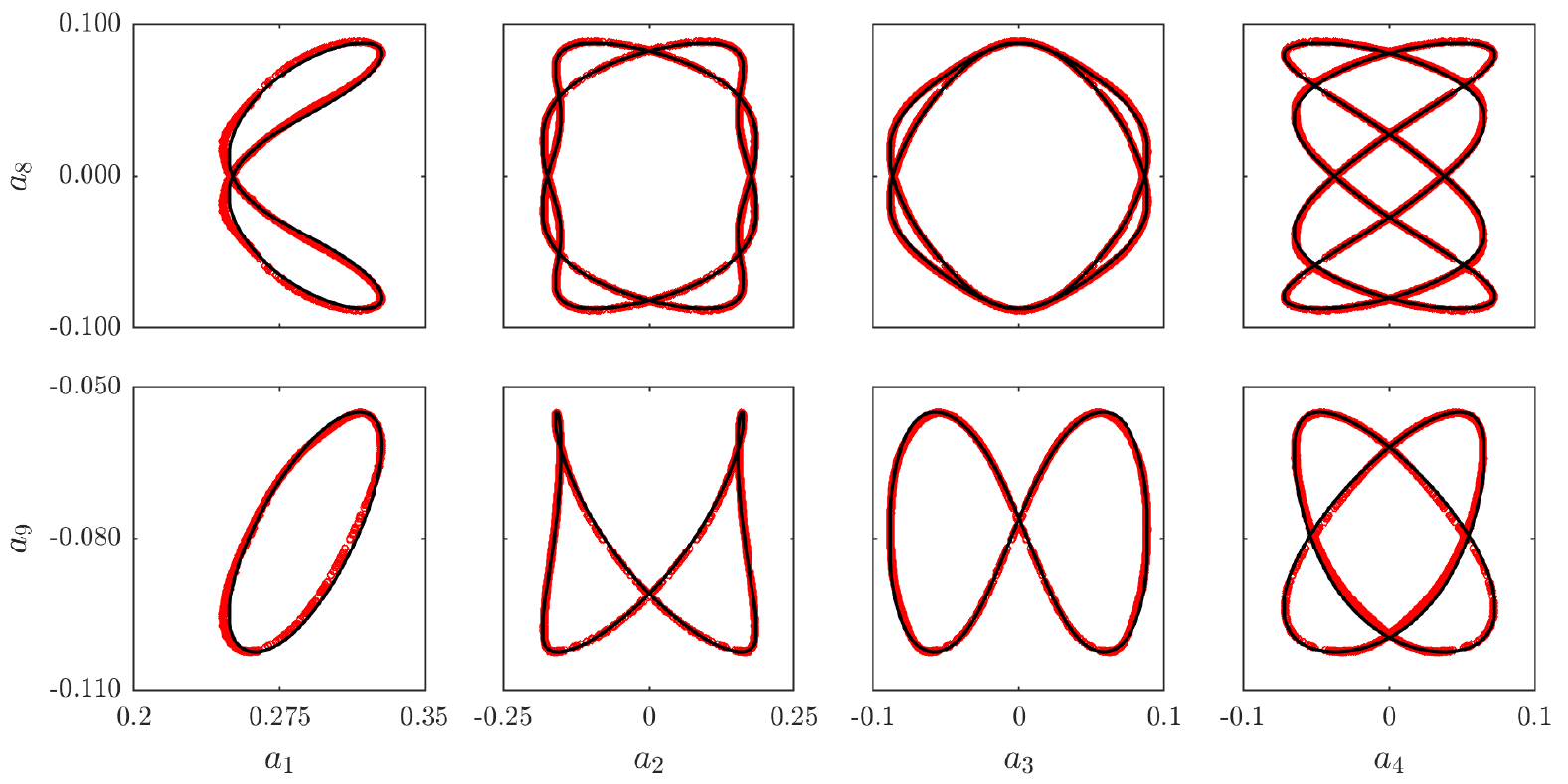}
	\caption{UPOs that maximize $\overline{\mathcal{E}}$ at $\Rey = 105$, projected onto various 2D subspaces ({\color{black}\solidrule}). Also shown are points in the set $\mathcal{S}_{3.93\mathrm{e}{-6}}$ (\mycircle{matlabred}), with the auxiliary function used of the form $P_{9}(\vec{a}) + \|\vec{a}\|^{10}$.}
	\label{fig:PO_Approx105}
\end{figure}
\begin{figure}
	\centering
	\includegraphics[width=0.98\textwidth]{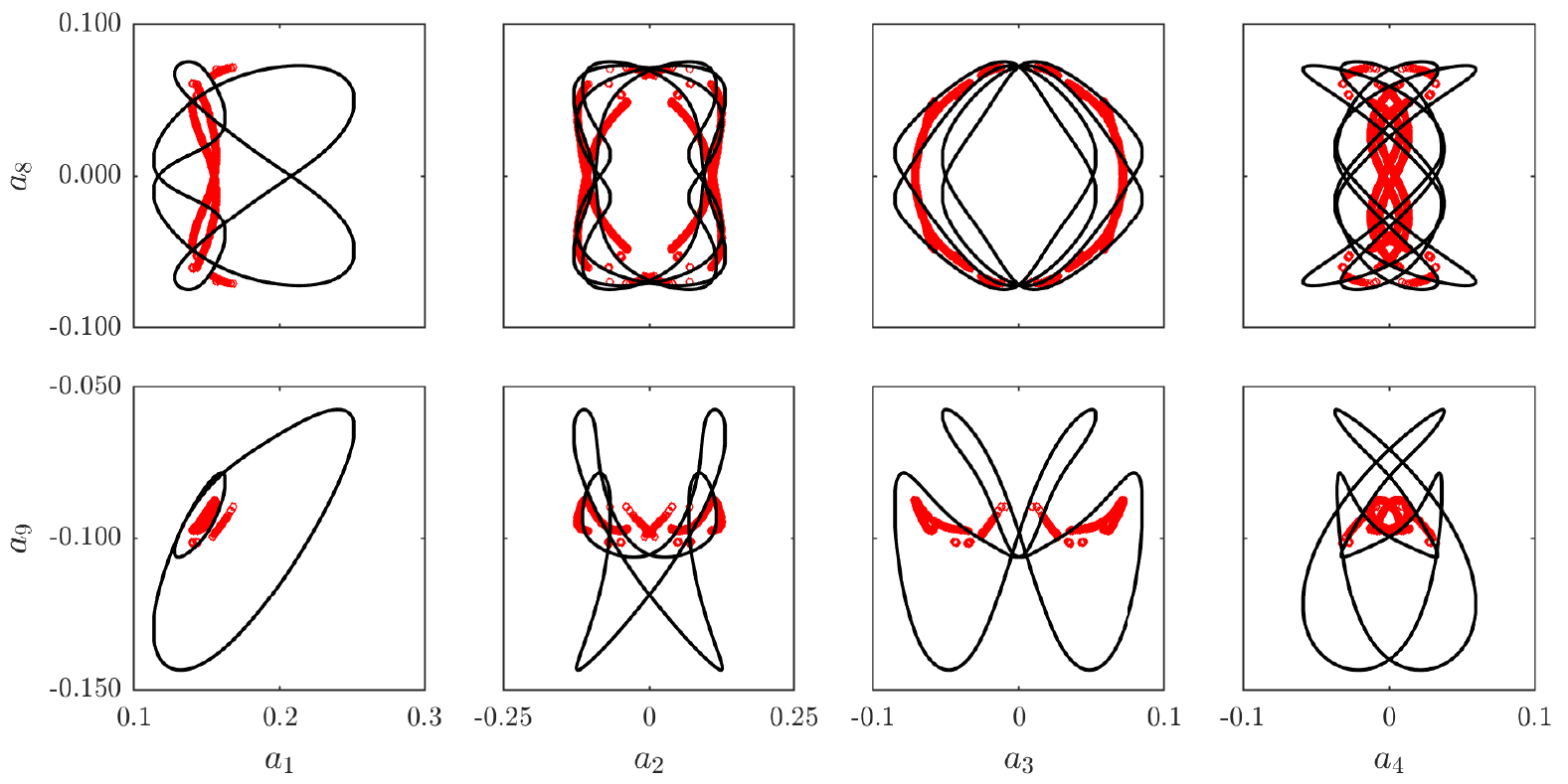}
	\caption{Most extremal known UPOs for $\overline{\mathcal{E}}$ at $\Rey = 250$, projected onto various 2D subspaces ({\color{black}\solidrule}). Also shown are points in the set $\mathcal{S}_{4.71\mathrm{e}{-6}}$ (\mycircle{matlabred}), with the auxiliary function used of the form $P_{9}(\vec{a}) + \|\vec{a}\|^{10}$.} 
	\label{fig:PO_Approx250}
\end{figure}

\section{Conclusions}
\label{s:conclusion}
In this paper we discussed the numerical implementation of an approach to approximate trajectories that maximize or minimize infinite-time averages of ODE systems governed by polynomial dynamics. The approach, originally proposed in~\cite{Tobasco2018}, relies on the construction of an auxiliary function $V$ that produces a near-sharp bound $\lambda$ on the infinite-time average of a quantity $\Phi$, followed by the computation of sublevel sets $\mathcal{S}_\varepsilon$ of the function $\lambda - \Phi - \vec{f}\cdot \nabla V$. Here, we have proved that arbitrarily optimal polynomial $V$, which yield arbitrarily sharp bounds on the extremal infinite-time average of $\Phi$, can always be constructed computationally when the optimization problem~\cref{e:inf-sup-V-poly} is implemented using weighted SOS constraints. Since such constraints can be reformulated as SDPs, the first step of the approximation procedure can always be carried out efficiently (precisely, in polynomial time) given enough computational resources. We have also proposed a simple and highly parallelizable scheme based on direct nonlinear minimization of $\lambda - \Phi - \vec{f}\cdot \nabla V$ to compute points in the set $\mathcal{S}_\varepsilon$ for high-dimensional systems, for which finding the whole set $\mathcal{S}_\varepsilon$ is impossible in practice. These points often approximate the extremal trajectories well, and provide good initial conditions to algorithms for converging UPOs.

The potential of these methods was demonstrated by applying them to a nine-dimensional model of shear flow~\cite{Moehlis2004}. We have calculated upper bounds on the mean energy of perturbations from the laminar state and lower bounds on the mean energy dissipation rate across a range of $\Rey$, and produced approximations to the corresponding extremal orbits. This has resulted in the discovery of three families of UPOs, born respectively at $\Rey = 81.24$, $\Rey = 241.5$ and $\Rey = 330.2$. These families are, to the best of our knowledge, new and not among those previously reported in~\cite{Moehlis2005}. The first family of UPOs is simultaneously minimal for $\overline{\mathcal{D}}$ and maximal for $\overline{\mathcal{E}}$ when $81.24 \lessapprox \Rey \lessapprox 125$, but cease to be extremal for both quantities at higher $\Rey$. The other two families of UPOs appear not to be extremal at any $\Rey$. Since increasing the degree of the auxiliary function $V$ brings very little improvement in our bounds for both $\overline{\mathcal{D}}$ and $\overline{\mathcal{E}}$ when $\Rey \gtrapprox 125$, we conjecture the existence of another family of UPOs that simultaneously optimizes both quantities and attains values that saturate the corresponding bounds. Finding these extremal orbits remains a topic for future work, although \cref{fig:PO_Approx250} gives an idea of the regions of state space in which they lie.

While we have focused on finding UPOs that optimize $\overline{\mathcal{E}}$ and $\overline{\mathcal{D}}$, other orbits can be discovered by varying the choice of $\Phi$. Indeed, any particular UPO can in principle be found by letting $\Phi$ be an indicator function, equal to unity on the UPO of interest and zero elsewhere, or a smooth polynomial approximation thereof. Of course, this cannot be done in practice unless a priori information on the location of a certain orbit is available. Nevertheless, varying $\Phi$ enables one to systematically search for periodic orbits guided by the SOS approximations. This much is true not only for the shear flow model considered here, but also for many other ODE systems. In addition, the methods can be extended to discrete-time and stochastic dynamics; see~\cite{Korda2018a} for a discussion of this from the perspective of invariant measures.

In principle, the techniques discussed in this work are applicable to polynomial ODEs of arbitrarily high dimension. At the time of writing, however, constructing near-optimal auxiliary functions using off-the-shelf software becomes prohibitively expensive for ODEs with more than $\mathcal{O}(10)$ states, because the computational resources required by standard interior-point SDP solvers grow quickly as the degree of $V$ is raised, and the numerical conditioning degrades. To help alleviate these problems, \cref{th:wsos-multipliers} provides a method to reduce the number of decision variables in the related SDPs by invoking symmetries. Alternatives to computationally intensive general-purpose interior-point SDP solvers for optimization with SOS constraints have been proposed recently~\cite{Papp2019,Zheng2017,Zheng2019a} and further gains are possible if (scaled) diagonally dominant SOS constraints are used~\cite{Ahmadi2019}. These are stronger than standard SOS conditions but can be implemented as linear or second-order cone programs instead of SDPs. We expect that the computational barrier to studying high-dimensional systems using SOS optimization will be removed in the near future using a combination of any of these very promising techniques.

Finally, we remark that the algorithm we used in this work to numerically converge UPOs is a simple single-shooting Newton--Raphson method. Our SOS approximations to UPOs can be made even more useful if one employs more sophisticated techniques, which extract more information from the approximation than a single point used as the initial guess. One option is to initialize a multiple-shooting Newton--Raphson method~\cite{Christiansen2017} using a selection of the approximation points obtained from the SOS computation. Another interesting possibility is to leverage a variational method for finding UPOs introduced by Lan \& Cvitanovi\'c~\cite{Lan2004} and further developed by Boghosian \textit{et al.}~\cite{Boghosian2011a}. This method finds a UPO by evolving an initial loop which approximates it. Connecting by smooth curves the approximation points obtained with the methods discussed in this work would generate such an initial loop, although it is not immediately clear how this connection should be performed. Whether the SOS approximation method presented here can successfully be utilized in conjunction with multiple-shooting or the variational approach of~\cite{Lan2004,Boghosian2011a} remains to be seen and should be investigated by future work. In principle, such a combined approach has the potential to become a powerful new method to tackle the long standing problem of searching for UPOs in differential dynamical systems.

\subsection*{Acknowledgment}
The authors would like to thank Jeff Moehlis for providing data used to produce \cref{fig:PE_Bounds}. An anonymous referee's suggestion on how one could generate random points with a favorable probability distribution constitutes part of the discussion in \cref{s:po-approximation}, and is gratefully acknowledged.

\bibliography{SIADS-references-v3}
\end{document}